\documentclass[fleqn]{article}
\usepackage{epsfig}
\usepackage{epstopdf}

\usepackage{latexsym}
\usepackage{graphicx}
\usepackage{latexsym, bm, amsthm, amsmath,amsfonts}
\usepackage{contour,soul,color}
\usepackage[colorlinks=true]{hyperref}
\usepackage{booktabs,cases}
\usepackage{tabularx,multirow} 
\usepackage{amsmath,amssymb}
\usepackage{amsthm}
\usepackage{amsfonts}
\numberwithin{equation}{section}

\newtheorem{theorem}{Theorem}[section]
\newtheorem{lemma}{Lemma}[section]

\newtheorem{remark}{Remark} [section]

\newtheorem{ass}{Assumption}[section]

\def\X{\mathcal{X}}

\def\Z{{\cal Z}}

\def \[{\begin{equation}}
\def \]{\end{equation}}

\def\L{{\cal L}}

\begin{document}

\title{\bf \Large Convergence analysis of the direct extension of ADMM for multiple-block separable convex minimization}
\author{ Min Tao\thanks{Department of
    Mathematics, Nanjing University, Nanjing, 210093, China. Email: \texttt{taom@nju.edu.cn}. This author was supported by the Natural Science Foundation of China: NSFC-11301280.}
    \/ \hskip 0.3cm  and\hskip 0.3cm
Xiaoming Yuan\thanks{Department of Mathematics, Hong Kong Baptist University, Hong Kong, China. Email: \texttt{xmyuan@hkbu.edu.hk}. This author is supported by a General Research Fund from Hong Kong Research Grants Council.}
}
\maketitle

\begin{abstract} Recently, the alternating direction method of multipliers (ADMM) has found many efficient applications in various areas; and it has been shown that the convergence is not guaranteed when it is directly extended to the multiple-block case of separable convex minimization problems where there are $m\ge 3$ functions without coupled variables in the objective. This fact has given great impetus to investigate various conditions on both the model and the algorithm's parameter that can ensure the convergence of the direct extension of ADMM (abbreviated as ``e-ADMM"). Despite some results under very strong conditions (e.g., at least $(m-1)$ functions should be strongly convex) that are applicable to the generic case with a general $m$, some others concentrate on the special case of $m=3$ under the relatively milder condition that only one function is assumed to be strongly convex. We focus on extending the convergence analysis from the case of $m=3$ to the more general case of $m\ge3$. That is, we show the convergence of e-ADMM for the case of $m\ge 3$ with the assumption of only $(m-2)$ functions being strongly convex; and establish its convergence rates in different scenarios such as the worst-case convergence rates measured by iteration complexity and the asymptotically linear convergence rate under stronger assumptions. Thus the convergence of e-ADMM for the general case of $m\ge 4$ is proved; this result seems to be still unknown even though it is intuitive given the known result of the case of $m=3$. Even for the special case of $m=3$, our convergence results turn out to be more general than the exiting results that are derived specifically for the case of $m=3$.
\end{abstract}



\section{Introduction}


We consider a canonical convex minimization model with separable structure and linear constraints, whose objective function is the sum of $m$ functions without coupled variables:
\begin{eqnarray}\label{mini3}
\begin{array}{l}
\min\left\{\sum\limits_{i=1}^m \theta_i(x_i)\;\Big|\; \sum\limits_{i=1}^mA_ix_i = b, \; x_i\in\X_i,\; i=1,2,\ldots,m \right\},
\end{array}
\end{eqnarray}
where $\theta_i : {\Re}^{n_i}\to {\Re}$
($i=1,2,\cdots,m$) are closed proper convex functions (not
necessarily smooth); $A_i \in {\Re}^{l \times n_i}$
($i=1,2,\cdots,m$); ${\cal X}_i \subseteq {\Re}^{n_i}$
($i=1,2,\cdots,m$) are nonempty closed convex sets; $b \in {\Re}^l$ and
$\sum_{i=1}^m n_i=n$.
The solution set of (\ref{mini3}) is assumed to be nonempty throughout our discussion.

Let the augmented Lagrangian function of (\ref{mini3}) be
\begin{eqnarray}\label{ALf3}
\begin{array}{c}
\mathcal{L}_{\beta}(x_1,x_2,\ldots, x_m,z):=\sum\limits_{i=1}^m\theta_i(x_i)- z^\top(\sum\limits_{i=1}^mA_ix_i-b)
+\frac{{\beta}}{2}\left\|\sum\limits_{i=1}^m A_ix_i-b\right\|^2
\end{array}
\end{eqnarray}
with $z \in {\Re}^l$ the Lagrange multiplier and $\beta>0$ the penalty parameter.
We focus on the following iterative scheme with $m\ge 3$:
\begin{subnumcases}{\label{EADMM}}\label{x1a}
x_1^{k+1}=\hbox{arg}\min\big\{\mathcal{L}_{\beta}(x_1,x_2^k,\ldots,x_m^k,z^k) \mid x_1\in\mathcal{X}_1\big\},\\
\label{x2a} x_2^{k+1}=\hbox{arg}\min\big\{\mathcal{L}_{\beta}(x_1^{k+1},x_2,\ldots,x^k_m,z^k) \mid x_2\in\mathcal{X}_2\big\},\\
\qquad\qquad\qquad\cdots\;\cdots\;\cdots\nonumber\\
\label{xia}x_i^{k+1}=\hbox{arg}\min\big\{\mathcal{L}_{\beta}(x_1^{k+1},\ldots,x_{i-1}^{k+1},x_i,x_{i+1}^k,\ldots,x^k_m,z^k) \mid x_i\in\mathcal{X}_i\big\},\\
\qquad\qquad\qquad\cdots\;\cdots\;\cdots\nonumber\\
\label{x3a} x_m^{k+1}=\hbox{arg}\min\big\{\mathcal{L}_{\beta}(x_1^{k+1},x_2^{k+1},\ldots,x_{m-1}^{k+1},x_m,z^k) \mid x_m\in\mathcal{X}_m\big\},\\ \label{lap}
z^{k+1}=z^k-\beta\left(\sum\limits_{i=1}^mA_ix_i^{k+1}-b\right),
\end{subnumcases}
which starts from an given iterate $(x_2^k,\ldots,x_m^k,z^k)$.
When $m=2$, the scheme (\ref{x1a})-(\ref{lap}) reduces to the alternating direction method of multipliers (ADMM) originally proposed in \cite{GM}. The convergence of ADMM has been well studied in the literature, see \cite{Gabay83,Gabay76,HeYuan-SINUM, HeYuan-NM}. Recently, ADMM has found many applications in a wide range of areas; we refer to, e.g., \cite{Boyd,Eck12,Glow14} for its review. For the generic case of $m\ge3$, the scheme (\ref{x1a})-(\ref{lap}) can be regarded as a direct extension of the alternating direction method of multipliers (abbreviated as ``e-ADMM").
Despite the inertia in algorithmic design and the numerical efficiency in empirical implementation (e.g., in \cite{PGWXM,TY,ZYY}), it was shown in \cite{CHYY} that the e-ADMM (\ref{x1a})-(\ref{lap}) is not necessarily convergent when $m=3$. By mathematical induction, it is easy to prove that the same conclusion for the general case of $m\ge 3$. This rather surprising fact has immediately given impetus to investigate various conditions to ensure the convergence of the scheme (\ref{x1a})-(\ref{lap}) with $m\ge 3$.


In the literature, there are some results for the generic case with a general $m>3$, it was shown in \cite{HanYuan-JOTA} that the scheme (\ref{x1a})-(\ref{lap}) is convergent if all the functions $\theta_i$ are strongly convex. In \cite{LMZa}, the global convergence of (\ref{x1a})-(\ref{lap}) was shown under the conditions that  $(m-1)$ of the functions $\theta_i$ are strongly convex. In \cite{LMZb}, the linear convergence of (\ref{x1a})-(\ref{lap}) was shown under the condition that at least $(m-1)$ of the functions $\theta_i$ are strongly convex together with other assumptions such as $\nabla \theta_i$ are  Lipschitz continuous and $A_i$ are full row/column rank. In addition, the authors of \cite{HongLuo} showed that the linear convergence can be guaranteed if the step size of the last step (\ref{lap}) for updating the multiplier $z^{k+1}$ is shrank by a sufficiently small factor and a certain error bound condition is satisfied.

For the special case of $m=3$, there is a richer set of literature. The first one is \cite{CSY}, which shows the convergence of (\ref{EADMM}) under the assumption that two functions of $\theta_i$ are strongly convex. Still requiring the strong convexity of two functions, the work \cite{LMZa} proves some refined convergence results such as the $O(1/t)$ ergodic convergence rate and $o(1/t)$ non-ergodic convergence rate measured by the iteration complexity, where $t$ is the iteration number. Later, the results in \cite{CSY,LMZa} were improved in \cite{CHY,LiSunToh}, in which the convergence of (\ref{x1a})-(\ref{lap}) was obtained with only one strongly convex function for the case $m=3$. According to the results in \cite{CHYY}, the strong convexity of at least one function seems minimal for the special case of $m=3$ of (\ref{mini3}) to ensure the convergence of (\ref{x1a})-(\ref{lap}); and the work in \cite{CHY,LiSunToh} verifies this conclusion positively.

Given the mentioned results for the case of $m =3$, by analogy, can we claim that we need $(m-2)$ strongly convex functions amid $\theta_i$'s to ensure the convergence of the scheme (\ref{x1a})-(\ref{lap}) for the generic case with a general $m$ that can be larger than $3$? Our main goal in this paper is to answer this question affirmatively. As we shall show, though the answer seems to be intuitive because of the known result for the special case of $m=3$, technically the extension from $m=3$ to $m\ge 3$ is highly nontrivial. One may ask if we can only require $(m-3)$ of the functions $\theta_i$ to be strongly convex to ensure the convergence of (\ref{EADMM}) when $m\ge 4$. In Section \ref{sect72}, we give an example to show  that in general it is not guaranteed and thus verify the rationale of considering the convergence of (\ref{EADMM}) with $m\ge 3$ with the assumption of $(m-2)$ functions being strongly convex. We also refer to the recent work \cite{DY} for a more general study in the operator context and it includes the case of (\ref{mini3}) with $m-2$ strongly convex functions as a special case. But the resulting algorithm (see Algorithm 9 in \cite{DY}) for this special case is not the same as the e-ADMM (\ref{EADMM}) under our consideration, e.g., for the subproblems accompanying the strongly convex functions, their objectives do not involve augmented Lagrangian terms and they are solved in parallel.

In addition to the strong convexity of some or all the functions in the objective of (\ref{mini3}), it is worthwhile to mention that the penalty parameter $\beta$ in (\ref{EADMM}) should be appropriately restricted to theoretically ensure the convergence, see, e.g., all the work \cite{CHY,CSY,HanYuan-JOTA,LiSunToh,LMZa,LMZb}\footnote{Note that such a requirement of the penalty parameter is usually conservative because it is used to sufficiently ensure the convergence. In numerical implementation, it can be appropriately relaxed to result in faster convergence.}. According to Theorem 4.1 in \cite{CHYY}, even all the functions $\theta_i$'s in the objective of (\ref{mini3}) are strongly convex, the scheme (\ref{x1a})-(\ref{lap}) with $m=3$ may be divergent if the penalty parameter $\beta$ is not well restricted. Similarly, in Section \ref{sect73}, we show that the scheme (\ref{EADMM}) could be divergent even when $(m-2)$ functions are strongly convex while the $\beta$ is not restricted appropriately. Therefore, to discuss the more difficult case where only some of the functions $\theta_i$'s are assumed to be strongly convex, we shall also restrict the penalty parameter into certain intervals when discussing the the convergence of e-ADMM (\ref{x1a})-(\ref{lap}) with $m\ge 3$. Indeed, as we shall elucidate, the range of $\beta$, which is eligible to the case with a generic $m$, is even larger than those in \cite{CHY,LiSunToh} when it reduces to the special case of $m=3$. That is, we shall prove the convergence for the e-ADMM (\ref{EADMM}) by requiring only $(m-2)$ strongly convex functions and a larger range of $\beta$ for $m\ge3$. Moreover, we shall establish the worst-case $O(1/t)$ convergence rate in the ergodic sense for $m\ge3$, where $t$ is the iteration counter; and explore some stronger conditions that can ensure the asymptotically linear convergence for $m\ge3$. Thus, compared with existing work in the same category such as \cite{CHY,CSY,HanYuan-JOTA,LiSunToh,LMZa,LMZb}, the convergence results in this paper are more general and they are proved under weaker conditions.

The rest of this paper is organized as follows. We summarize some notation, present the assumptions for future discussion and recall some known results in Section 2. In Section 3, we prove the convergence of e-ADMM (\ref{x1a})-(\ref{lap}) under certain assumptions; this is the main result of this paper. Then, we establish the worst-case convergence rate measured by the iteration complexity in Section 4. In Section 5, we show that the scheme (\ref{x1a})-(\ref{lap}) can be guaranteed to be globally linear convergent if further conditions are posed. In Section 6, we show that the convergence of e-ADMM (\ref{EADMM}) may not be guaranteed if there are no appropriate assumptions on the model (\ref{mini3}) or the penalty parameter $\beta$ in (\ref{EADMM}). Some examples are also constructed. Finally, we draw some conclusions in Section 7.
\medskip

\section{Preliminaries}

In this section, we define some notation to be used; present some assumptions on the model (\ref{mini3}) under which our convergence analysis will be conducted; show the optimality condition of the model (\ref{mini3}) in the variational inequality context.

\subsection{Notation}

The domain of a function $f$ is denoted by $dom(f)$ and the set of all relative interior points of a given nonempty convex set
$\Omega$ by $ri(\Omega)$. Given a vector $x\in\Re^n$, the notation $x_{[i:j]}$ $(1\le i\le j\le n)$ denotes the subvector of $x$ consisting of the $i$-th up to the $j$-th entries of $x$.
If $i=j$, $x_{[i]}$ just denotes the $i$-th entry of $x$.
For any given vector $x$ and a symmetric positive semi-definite matrix $M$ with appropriate dimensionality, we use $\|x\|_M^2$ to denote $x^\top Mx$.
For an symmetric matrix $M$, let $\|M\|$ denote its $2$-norm. If $M$ is nonsymmetric, we use $\|M\|:=\sqrt{\|M^\top M\|}$ and $\rho(M)$ to denote its spectral radius, i.e., the maximal absolute value of its eigenvalues. A function $f: {\Re}^n \to (-\infty, \infty]$ is strongly convex with modulus $\mu>0$ if it satisfies
\[\label{strongconvex}
f(tx+(1-t)y)\le tf(x)+(1-t)f(y)-\frac{\mu}{2}t(1-t)\|x-y\|^2,\;\;\forall x,y\in {\Re}^n;\]
where $t \in [0,1]$.

Then, based on the coefficient matrices $A_i$ in (\ref{mini3}) and the penalty parameter $\beta$ in (\ref{EADMM}), we define some matrices to simplify our notation in later analysis. More specifically, for $m\ge3$, let the block  triangular matrices $M$, $N$ and  block diagonal matrix $Q$ be respectively defined as:
\[\label{MatrixM}
M=\left (\begin{array}{cccccc}
 0&A_1^\top A_2&A_1^\top A_3&\cdots&A_1^\top  A_m&0\\
 0&0&A_2^\top A_3&\cdots&A_2^\top A_m&0\\
 \vdots&\cdots&\ddots&\vdots&\vdots&0\\
 0&0&0&0&A_{m-1}^\top A_m&0\\
 0&0&0&0&0&0\\
 0&0&0&0&0&0
 \end{array}\right),\;
 N=\left (\begin{array}{ccccc}
 0&0&0&\cdots&0\\
 0&A_2^\top A_2&\cdots&\cdots&\vdots\\
 \vdots&\vdots&\ddots&\vdots&\vdots\\
0&A_m^\top A_2&\cdots&A_{m}^\top A_m&0\\
 0&0&0&0&0
 \end{array}\right),
 \]
 and
\[\label{Q}
Q:=\left(\begin{array}{ccccc}\beta A_2^\top A_2&0&\cdots&0&0\\
0&\beta A_3^\top A_3&\cdots&0&0\\
\vdots&\vdots&\ddots&\vdots&\vdots\\
0&\cdots&\cdots&\beta A_m^\top A_m&0\\
0&\cdots&\cdots&0&\frac{1}{\beta}I\end{array}
\right).\]
Note that both $M$ and $N$ are in the space ${{(n+l)}\times{(n+l)}}$; and $Q$ in ${{(\sum_{i=2}^m n_i+l)}\times{(\sum_{i=2}^m n_i+l)}}$. Also, the matrix $Q$ defined in (\ref{Q}) is positive definite
if $A_i$ ($i=2,\ldots,m$) are assumed to be full column rank and $\beta>0$.

\subsection{Assumptions}

Then, we present the assumptions on the model (\ref{mini3}) to conduct the convergence analysis for the e-ADMM (\ref{EADMM}) with a general $m\ge 3$.

\begin{ass}\label{AS1}
In \eqref{mini3} with $m \ge 3$, the functions $\theta_1$ and $\theta_2$ are convex; the functions $\theta_3,\ldots\theta_m$ are strongly convex with the modulus $\mu_i>0$ $(i=3,\ldots,m)$; $A_i$ ($i=1,\ldots,m$) are full column rank matrices.
\end{ass}


\begin{ass}\label{AS3}
There exists $u'=(x_1',\ldots,x_m')\in ri(dom(\theta_1)\times dom(\theta_2)\times\cdots\times dom(\theta_m))\cap\cal F$,
where
\begin{eqnarray}\nonumber
{\cal F}:=\left\{ u=(x_1,\ldots,x_m)\in{\cal X}_1\times\cdots\times{\cal X}_m|\sum_{i=1}^m A_i x_i=b\right\}.
\end{eqnarray}
\end{ass}
Note that both $\theta_1$ and $\theta_2$ are assumed to be convex; but we also say that they both satisfy the inequality (\ref{strongconvex}) with $\mu=0$ as long as there is no confusion. This helps us present the analysis in a unified notation.

\subsection{Optimality condition of (\ref{mini3}) as a variational inequality}
In the following, we characterize the optimality condition of the model (\ref{mini3}) as a variational inequality. The variational inequality representation plays a crucial role in our convergence analysis to be conducted.

First, let ${\cal W}:= \X_1 \times \X_2\times\cdots\times\X_m \times \Z$ and the Lagrangian function of \eqref{mini3} be
\begin{eqnarray}\nonumber
\L(x_1,x_2,\ldots,x_m,z):=\sum_{i=1}^m\theta_i(x_i)- z^\top( \sum_{i=1}^m A_i x_i-b),
\end{eqnarray}
with $z$ the Lagrange multiplier. Under Assumption \ref{AS3}, it follows from \cite[Corollary 28.2.2]{R} and \cite[Corollary 28.3.1]{R} that
the set of saddle points of $\L(x_1,x_2,\ldots,x_m,z)$, denoted by $\cal W^*$, is nonempty due to the nonempty assumption on the solution set of (\ref{mini3}). Then, solving \eqref{mini3} amounts to finding a saddle point
of $\L(x_1,x_2,\ldots, x_m,z)$.
Therefore, the optimality condition of the model (\ref{mini3}) can be characterized by finding $w^*\in{\cal W}^*$ such that:
\[ \label{MA-VIxyl}
    \left\{
    \begin{array}{l}
     \theta_1(x_1) - \theta_1(x_1^*) + (x_1- x_1^*)^\top (- A_1^\top z^*) \ge 0, \;\;\;\\
      \theta_2(x_2) - \theta_2(x_2^*) + (x_2- x_2^*)^\top (- A_2^\top z^*) \ge 0, \;\;\;\\
     \qquad \cdots \cdots  \quad \qquad\qquad \cdots \cdots \\
       \theta_i(x_i) - \theta_i(x_i^*) + (x_i- x_i^*)^\top (- A_i^\top z^*) \ge \frac{\mu_i}{2}\|x_i-x_i^*\|^2,\;\;\;\\
       \qquad\cdots \cdots  \quad  \qquad\qquad \cdots  \cdots\\
     \theta_m(x_m) - \theta_m(x_m^*) +(x_m- x_m^*)^T ( - A_m^\top z^*) \ge \frac{\mu_m}{2}\|x_m-x_m^*\|^2, \;\;\;\\
     \sum_{i=1}^m A_ix_i^* - b =0,
      \end{array} \right.
       \;\;\forall\; w\in{\cal W},
    \]
where $\mu_i$ is  the strongly convex modulus of $\theta_i$ for $i=3,\ldots,m$. More compactly, the system (\ref{MA-VIxyl}) can be written as the variational inequality:
\begin{subequations}\label{MSVII}
\[\label{MSVI-F}
   \hbox{VI}({\cal W},F,\theta) \qquad \theta(u) - \theta(u^*)+  (w-w^*)^\top F(w^*) \ge \sum_{i=3}^m\frac{\mu_i}{2}\|x_i - x_i^*\|^2, \quad \forall \
       w\in  {\cal W}, \qquad \qquad
       \]
where
\[  \label{MD-F}
  u=\left(  \begin{array}{c} x_1 \\ x_2
                            \\ \vdots\\ x_m
                           \end{array}
                            \right),
     \quad
      w=\left(  \begin{array}{c} x_1 \\ x_2
                            \\ \vdots\\ x_m\\
                            z \end{array}
                            \right) \quad \hbox{and} \quad
     \theta(u) =\sum_{i=1}^m\theta_i(x_i),\quad
   F(w)=\left(  \begin{array}{c}-A_1^\top z \\
                            -A_2^\top  z
                            \\ \vdots\\-A_m^\top z\\
                            \sum_{i=1}^m A_ix_i - b \end{array}
                            \right).
    \]
\end{subequations}
Note that $u$ collects all the primal variables in
(\ref{mini3}) and it is a sub-vector of $w$. Since the variable $x_1$ is not involved in the iteration of e-ADMM (\ref{EADMM}), we denote by
$$
 v=(x_2, x_3, \cdots, x_m, z)
 $$
all the primal and dual variables that are essentially involved in the iteration (\ref{EADMM}). Moreover,  the
solution set of $\hbox{VI}({\cal W},F,\theta)$, i.e., ${\cal W}^*$, is also convex due to Theorem 2.3.5 in \cite{Fa-Pang}.
Accordingly, we also use the notation
\begin{eqnarray}\nonumber
{\cal V}^* = \{ (x_2^*, \ldots, x_m^*,\lambda^*) \, | \,
        (x_1^*,x_2^*, \ldots, x_m^*,\lambda^*)\in {\cal W}^* \}.
\end{eqnarray}

\subsection{Two elementary lemmas}

In the following, we present two elementary lemmas. The first one is trivial and its proof is omitted.
\begin{lemma}\label{F-monotone}
The mapping $F(w)$ defined in (\ref{MD-F}) satisfies
\[   \label{MonoF}
(w'-w)^\top(F(w')-F(w)) = 0 , \quad \forall \,w',  w\in \Re^{n+l}.
\]
\end{lemma}

The second lemma shows that the spectral radius is a continuous function with respect to the 2-norm of a matrix. We shall use this property in Section \ref{ext}.

\begin{lemma}\label{conspec} Given two matrices $A\in\Re^{n\times n}$ and $\Delta\in\Re^{n\times n}$ that are not necessarily symmetric.
Suppose that $\|\Delta\|<1$. Then, there exists a positive constant $C$ depending only on the matrix $A$ such that
\begin{eqnarray}\label{contispectral}
|\rho(A+\Delta)-\rho(A)|\le C\|\Delta\|.
\end{eqnarray}
\end{lemma}
\begin{proof}
First, using the triangle inequality, we have
$$
\left|\|(A+\Delta)^\top (A+\Delta)\|-\|A^\top A\| \right|\le \|\Delta^\top A+A^\top \Delta+\Delta^\top \Delta\| \le(2\|A\|+\|\Delta\|)\|\Delta\|\le(2\|A\|+1)\|\Delta\|.
$$
Then, it follows from the definition of the spectral radius $\rho(\cdot)$ that
\begin{eqnarray}
\left|\rho(A+\Delta)-\rho(A)\right|&=&\left| \sqrt{\|(A+\Delta)^\top (A+\Delta)\| }-\sqrt{ \|A^\top A\| }\right| =\left|\frac{ \|(A+\Delta)^\top (A+\Delta)\| -  \|A^\top A\|}{\sqrt{\|(A+\Delta)^\top (A+\Delta)\| }+\sqrt{ \|A^\top A\| }} \right|\nonumber\\
&\le&\frac{1}{\sqrt{ \|A^\top A\| }}(2\|A\|+1)\|\Delta\|.
\end{eqnarray}
Thus, the inequality (\ref{contispectral}) holds with $C:=\frac{1}{\sqrt{ \|A^\top A\| }}(2\|A\|+1)$ which is positive and only dependent on the matrix $A$.
\end{proof}

\section{Convergence}\label{conv}

In this section, we prove the convergence of the e-ADMM \eqref{EADMM} for $m\ge3$ under the mentioned assumptions on the model (\ref{mini3}) with a certain restriction on the penalty parameter $\beta$. This is the main result of this paper. As mentioned, the proof is highly nontrivial. So we organize the discussion into several subsections. The roadmap of the proof is also reflected by the titles of these subsections.

\subsection{Discerning the difference of an iterate from a solution point}

We intend to observe the iterate $w^{k+1}$ generated by the e-ADMM \eqref{EADMM} and quantity its difference from a solution point in ${\cal W}^*$ in terms of the variational inequality representation (\ref{MSVII}) of the optimality condition. Since the iterate $w^{k+1}$ generated by the scheme \eqref{EADMM} can be expressed in the variational inequality form, it is possible to compare it with the variational inequality representation (\ref{MSVII}) of the optimality condition of the model (\ref{mini3}) and so discern the difference of the iterate $w^{k+1}$ from a solution point in ${\cal W}^*$. More precisely, we can show that this difference can be measured by some crossing terms and hence we need to carefully analyze these crossing terms. The following lemma follows from the first-order optimality conditions of the subproblems in the e-ADMM (\ref{EADMM}).

\begin{lemma}\label{lemma31}
Suppose Assumptions \ref{AS1} and \ref{AS3} hold. Let $\{w^k\}$ be the sequence generated by the e-ADMM \eqref{EADMM}. Then, we have
$x_i^{k+1}\in{\cal X}_i$ $(i=1,\ldots,m)$ and
\begin{eqnarray}\label{x3e2}
\theta_i(x_i)-\theta_i(x_i^{k+1})-
(x_i-x_i^{k+1})^\top A_i^\top[z^{k+1}-\beta \sum_{j=i+1}^m A_j(x_j^k-x_j^{k+1})]\ge \frac{\mu_i}{2}\|x_i-x_i^{k+1}\|^2,                                                                        \end{eqnarray}
with $\mu_{1}=0$, $\mu_2=0$ and $\mu_i >0$ for $i=3,\ldots,m$.
\end{lemma}
\begin{proof} According to the optimality condition of the $x_i$-subproblem (\ref{xia}), we have $x^{k+1}_i\in{\cal X}_i$ such that
\begin{eqnarray}
\theta_i(x_i)-\theta_i(x_i^{k+1})-
(x_i-x_i^{k+1})^\top A_i^\top[z^k-\beta (\sum_{j=1}^i A_j x_j^{k+1} +\sum_{j=i+1}^m A_j x_j^k-b)]\ge \frac{\mu_i}{2}\|x_i-x_i^{k+1}\|^2,\forall\;x_i\in{\cal X}_i. \nonumber
\end{eqnarray}
Substituting the equation \eqref{lap} into the last inequality, we obtain the assertion (\ref{x3e2}).
\end{proof}

Recall the characterization of ${\cal W}^*$ in (\ref{MSVII}). The following lemma reflects the discrepancy of $w^{k+1}$ from a solution point in ${\cal W}^*$.
\begin{lemma} \label{Lemma32}
Suppose Assumptions \ref{AS1} and \ref{AS3} hold. Let $\{w^k\}$ be the sequence generated by the e-ADMM \eqref{x1a}-\eqref{lap}. Then, we have
\begin{eqnarray}\label{Lem1}
 &&\theta(u)-\theta(u^{k+1})+(w-w^{k+1})^\top F(w^{k+1}) +\frac{1}{\beta}(z-z^{k+1})^\top (z^{k+1}-z^k)+\beta(w- w^{k+1})^\top M({w}^{k}-w^{k+1})\nonumber \\
 &&\ge
 \sum_{i=3}^m\frac{\mu_i}{2}\|x_i^{k+1}-x_i\|^2, \;\; \forall w\in {\cal W}.
\end{eqnarray}
\end{lemma}
\begin{proof}
First, it follows from (\ref{lap}) that
\begin{eqnarray}
\label{slz}(z-z^{k+1})^\top [\sum_{i=1}^m A_i x_i^{k+1}-b-\frac{1}{\beta}(z^k-z^{k+1})]\ge0, \;\; \forall z \in \Re^l.
\end{eqnarray}
Combining the inequalities (\ref{x3e2}) for $i=1,\ldots,m$, with the above inequality, we have
\begin{eqnarray}\label{optint}
\left\{\begin{array}{l}
\theta_1(x_1)-\theta_1(x_1^{k+1})-(x_1-x_1^{k+1})^\top A_1^\top[z^{k+1}-\beta \sum_{j=2}^m A_j(x_j^k-x_j^{k+1})] \ge 0,\\
\theta_2(x_2)-\theta_2(x_2^{k+1})-(x_2-x_2^{k+1})^\top A_2^\top[ z^{k+1}-\beta \sum_{j=3}^m A_j(x_j^k-x_j^{k+1})]\ge 0,\\
\qquad\qquad\qquad\cdots\;\cdots\;\cdots\qquad\cdots\;\cdots\;\cdots\\
\theta_i(x_i)-\theta_i(x_i^{k+1})-(x_i-x_i^{k+1})^\top A_i^\top[z^{k+1}-\beta \sum_{j=i+1}^m A_j(x_j^k-x_j^{k+1})]\ge \frac{\mu_i}{2}\|x_i-x_i^{k+1}\|^2,\\
\qquad\qquad\qquad\cdots\;\cdots\;\cdots\qquad\cdots\;\cdots\;\cdots\qquad\qquad\qquad\qquad \qquad\qquad i=3,\ldots,m,\\
(z-z^{k+1})^\top [\sum_{i=1}^m A_i x_i^{k+1}-b-\frac{1}{\beta}(z^k-z^{k+1})]\ge0.
\end{array}\right.\forall \; w\in{\cal W}.
\end{eqnarray}
Adding all these inequalities together and using the definitions of $F$ in (\ref{MD-F}) and $M$ in (\ref{MatrixM}), we immediately obtain the assertion (\ref{Lem1}).
\end{proof}

\subsection{Replacing the crossing terms by summable quadratic terms}

According to Lemma \ref{Lemma32} and the optimality condition (\ref{MSVII}), it is clear that our emphasis should be analyzing the crossing term
\begin{eqnarray}\label{crossing}
\frac{1}{\beta}(z-z^{k+1})^\top (z^{k+1}-z^k)+\beta(w- w^{k+1})^\top M({w}^{k}-w^{k+1})
 \end{eqnarray}
which gives the difference of the iterate $w^{k+1}$ from a solution point in ${\cal W}^*$.
As we shall show later, the first term in (\ref{crossing}) can be handled easily, whereas the second one should be sophisticatedly treated. This is indeed the most technical part in the paper. We start from the following lemma.

\begin{lemma}\label{lemma31} Suppose Assumptions \ref{AS1} and \ref{AS3} hold. For the iterative sequence $\{w^k\}$ generated by the e-ADMM \eqref{EADMM}, we have
\begin{eqnarray}\label{mono3}
&&\left\langle z^{k+1}-  z^k,A_ix_i^{k+1}-A_ix_i^k \right\rangle\ge-\beta (A_i x_i^{k+1}-A_i x_i^k)^\top [\sum_{j=i+1}^mA_j (x_j^{k+1}-x_j^k)-\sum_{j=i+1}^m A_j(x_j^k-x_j^{k-1})]\nonumber\\
&&+\mu_i\|x_i^k-x_i^{k+1}\|^2 ,\;\; \;\;\;\;i=1,\ldots,m,
\end{eqnarray}
where $\mu_{1}=0$, $\mu_2=0$ and $\mu_i >0$ $(i=3,\ldots,m)$.
\end{lemma}

\begin{proof}
Setting $x_i:=x_i^k$ in \eqref{x3e2}, we get
\begin{eqnarray}\nonumber
\theta_i(x^k_i)-\theta_i(x_i^{k+1})-(x_i^k -x_i^{k+1})^\top A_i^\top[z^{k+1}+\beta\sum_{j=i+1}^m A_j(x_j^{k+1}-x_j^k)]\ge\frac{\mu_i}{2}\|x_i^{k+1}-x_i^k\|^2.
\end{eqnarray}
Setting $x_i:=x_i^{k+1}$ in \eqref{x3e2} with $k:=k-1$, we have
\begin{eqnarray}\nonumber
\theta_i(x^{k+1}_i)-\theta_i(x_i^{k})-(x_i^{k+1} -x_i^{k})^\top A_i^\top[z^{k}+\beta\sum_{j=i+1}^m A_j(x_j^{k}-x_j^{k-1})]\ge\frac{\mu_i}{2}\|x_i^{k+1}-x_i^k\|^2.
\end{eqnarray}
Adding the above two inequalities, we obtain that for  $i=1,\ldots,m$, it holds
\begin{equation}\label{xme2}
\langle A_ix_i^{k+1}-A_ix_i^k,  z^{k+1}-  z^k\rangle \ge \mu_i\|x_i^k-x_i^{k+1}\|^2-\beta (A_i x_i^{k+1}-A_i x_i^k)^\top [\sum_{j=i+1}^m A_j (x_j^{k+1}-x_j^k)-\sum_{j=i+1}^m A_j(x_j^k-x_j^{k-1})].
\end{equation}
Note we use the convention $\sum_{i=m+1}^m a_i=0$. The assertion (\ref{mono3}) is proved.
\end{proof}

\begin{lemma}\label{lemma312} Suppose Assumptions \ref{AS1} and \ref{AS3} hold. For the iterative sequence $\{w^k\}$ generated by the e-ADMM \eqref{EADMM}, we have
\begin{eqnarray}\label{monowhole}
\langle z^k-z^{k+1}, \sum_{i=2}^m A_i (x_i^k-x_i^{k+1})\rangle &\ge& -\beta \sum_{i=2}^{m-1} (A_i x_i^{k+1}-A_i x_i^k)^\top[\sum_{j=i+1}^m A_j (x_j^{k+1}-x_j^k)-\sum_{j=i+1}^m A_j(x_j^k - x_j^{k-1})] \nonumber\\
&& +\sum_{i=3}^m \mu_i\|x_i^{k+1}-x_i^k\|^2\end{eqnarray}
\end{lemma}
\begin{proof}
Adding inequalities (\ref{xme2}) from $i=2$ to $m$ together,
the assertion \eqref{monowhole} follows immediately.
\end{proof}

In the following lemma, we use the results in Lemmas \ref{Lemma32} and \ref{lemma312}; and represent the difference between the iterate $w^{k+1}$ from a solution point in ${\cal W}^*$ by some quadratic terms (see (\ref{deltAi}) and (\ref{deltzk})) and  crossing terms in terms of only $A_ix_i^{k+1}$ (see (\ref{Up1}) and (\ref{Up2})). This refined treatment turns out to be more convenient for successive operations over different subproblems; and it is the key to the proof of the main convergence results to be conducted.

\medskip
\begin{lemma}\label{lemma33}
Suppose Assumptions \ref{AS1} and \ref{AS3} hold. Let $\{w^k\}$ be the sequence generated by the e-ADMM \eqref{EADMM}. Then, for any $w\in {\cal W}$, we have
\begin{eqnarray}\label{convsol}
&&\theta(u)-\theta(u^{k+1})+(w-w^{k+1})^\top F(w^{}) +\beta(\sum_{i=1}^m A_i x_i-b)^\top \sum_{i=1}^m A_i (x_i^k-x_i^{k+1})\nonumber\\
&&\ge \frac{\beta}{2}\sum_{i=2}^m \triangle(x_i^{k+1}, x_i^{k}, x_i)+\frac{1}{2\beta}\triangle (z^{k+1},z^k,z) +\sum_{i=3}^m \left(\mu_i \|x_i^{k+1}-x_i^k\|^2+ \frac{\mu_i}{2}\|x_i^{k+1}-x_i\|^2\right)\nonumber\\
&&+ \Upsilon_1^{k+1} (x_i^{k+1},x_i^k,x_i^{k-1})+\Upsilon_2^{k+1} (x_i^{k+1},x_i^k,x_i),\nonumber\\
\end{eqnarray}
where
\begin{align}
&\triangle(x_i^{k+1}, x_i^{k}, x_i):=\|A_i x_i^{k+1}-A_i x_i\|^2-\|A_i x_i^k -A_i x_i\|^2+\|A_i x_i^k-A_i x_i^{k+1}\|^2,\label{deltAi}\\\nonumber\\
&\triangle(z^{k+1},z^{k},z):=\|z^{k+1}-z\|^2-\|z^k -z\|^2+\|z^k-z^{k+1}\|^2,\label{deltzk}\\\nonumber\\
&\Upsilon_1^{k+1} (x_i^{k+1},x_i^k,x_i^{k-1}):=-\beta \sum_{i=2}^{m-1}(A_i x_i^{k+1}-A_i x_i^k)^\top \left( \sum_{j=i+1}^m A_j (x_j^{k+1} -x_j^k)-  \sum_{j=i+1}^mA_j(x_j^k - x_j^{k-1}) \right),\label{Up1}\\\nonumber\\
&\Upsilon_2^{k+1} (x_i^{k+1},x_i^k,x_i):=\beta\sum_{i=3}^{m}\sum_{j=2}^{i-1}(A_i x_i^{k+1}-A_i x_i)^\top (A_j x_j^{k+1}-A_j x_j^k).\label{Up2}
\end{align}

\end{lemma}

\begin{proof}
First, using the definitions of $M$ and $N$ in (\ref{MatrixM}), it holds  that
\begin{eqnarray}\label{consolp1}
&&\beta(w-w^{k+1})^\top M(w^k-w^{k+1})+\beta(w-w^{k+1})^\top N(w^k-w^{k+1})\nonumber\\
&&=\beta\left( \sum_{i=1}^m(A_ix_i-b)-(A_i x_i^{k+1}-b))\right)^\top \sum_{i=2}^m(A_i x_i^k -A_i x_i^{k+1}).
\end{eqnarray}
Substituting (\ref{consolp1}) into the left-hand side of (\ref{Lem1}),  we obtain
\begin{eqnarray}\label{lem33p0}
&&\theta(u)-\theta(u^{k+1})+(w-w^{k+1})^\top F(w^{k+1}) +\frac{1}{\beta}(z-z^{k+1})^\top (z^{k+1}-z^k)\nonumber\\
&&+\beta(\sum_{i=1}^m A_i x_i -b)^\top \sum_{i=2}^m (A_i x_i^k - A_i x_i^{k+1})-\beta(\sum_{i=1}^{m}A_i x_i^{k+1}-b)^\top \sum_{i=2}^m (A_i x_i^k -  A_i x_i^{k+1})\nonumber\\
&&-\beta(w-w^{k+1})^\top N(w^k-w^{k+1})\ge \sum_{i=3}^m\frac{\mu_i}{2}\|x_i^{k+1}-x_i\|^2, \; \forall w\in {\cal W}.
\end{eqnarray}
On the other hand, using the definition of $\triangle(z^{k+1},z^{k},z)$ in (\ref{deltzk}), we have
\begin{eqnarray}\nonumber
\frac{1}{\beta}(z-z^{k+1})^\top (z^{k+1}-z^k)=-\frac{1}{2\beta}\triangle(z^{k+1},z^{k},z).\end{eqnarray}
Substituting the above identity into (\ref{lem33p0}) and using (\ref{x3a}), we obtain
\begin{eqnarray}\label{lem33p1}
&&\theta(u)-\theta(u^{k+1})+(w-w^{k+1})^\top F(w^{k+1})+\beta(\sum_{i=1}^m A_i x_i -b)^\top \sum_{i=2}^m( A_i x_i^k - A_i x_i^{k+1})\nonumber\\
&&\ge
\sum_{i=3}^m\frac{\mu_i}{2}\|x_i -x_i^{k+1}\|^2 +\beta(w-w^{k+1})^\top N(w^k-w^{k+1})+\frac{1}{2\beta}\triangle(z^{k+1},z^{k},z)\nonumber\\
&&+(z^k-z^{k+1})^\top \sum_{i=2}^m (A_i x_i^k
- A_i x_i^{k+1}), \;\forall\; w\in {\cal W}.
\end{eqnarray}
Next, substituting (\ref{monowhole}) into the last term of the right-hand of (\ref{lem33p1}),  it yields
\begin{eqnarray}\label{lem33p2}
&&\theta(u)-\theta(u^{k+1})+(w-w^{k+1})^\top F(w^{k+1})+\beta(\sum_{i=1}^m A_i x_i -b)^\top \sum_{i=2}^m( A_i x_i^k - A_i x_i^{k+1})\nonumber\\
&&\ge\sum_{i=3}^m\left[\frac{\mu_i}{2}\|x_i -x_i^{k+1}\|^2+\mu_i\|x_i^{k+1}-x_i^k\|^2\right]+\beta(w-w^{k+1})^\top N(w^k-w^{k+1})\nonumber\\
&&+\frac{1}{2\beta}\triangle(z^{k+1},z^{k},z)-\beta \sum_{i=2}^{m-1}(A_i x_i^{k+1} - A_i x_i^k)^\top\left [\sum_{j=i+1}^mA_j(x_j^{k+1}-x_j^k)-\sum_{j=i+1}^m A_j(x_j^{k}-x_j^{k-1})\right], \;\forall\; w\in{\cal W}.\nonumber\\
\end{eqnarray}
On the other hand, it follows from (\ref{deltAi}) and the definition of the matrix $N$ in (\ref{MatrixM}) that
\begin{eqnarray}\label{lem33p3}
&&\beta(w-w^{k+1})^\top N(w^k-w^{k+1})\nonumber\\
&&=\frac{\beta}{2}\sum_{i=2}^m \triangle(A_ix_i^{k+1}, A_ix_i^{k},A_i x_i) +\beta\sum_{i=3}^m\sum_{j=2}^{i-1}(A_i x_i^{k+1}-A_i x_i)^\top (A_j x_j^{k+1}-A_j x_j^k).
\end{eqnarray}
Substituting (\ref{lem33p3}) into (\ref{lem33p2}), we get
\begin{eqnarray}\label{lem33f}
&&\theta(u)-\theta(u^{k+1})+(w-w^{k+1})^\top F(w^{k+1})+\beta(\sum_{i=1}^m A_i x_i -b)^\top \sum_{i=2}^m( A_i x_i^k - A_x x_i^{k+1})\nonumber\\
&&\ge\sum_{i=3}^m\left[\frac{\mu_i}{2}\|x_i -x_i^{k+1}\|^2+\mu_i\|x_i^{k+1}-x_i^k\|^2\right]+\frac{\beta}{2}\sum_{i=2}^m \triangle(A_ix_i^{k+1}, A_ix_i^{k},A_i x_i)+\frac{1}{2\beta}\triangle(z^{k+1},z^{k},z)\nonumber\\
&&-\beta \sum_{i=2}^{m-1}(A_i x_i^{k+1} - A_i x_i^k)^\top \left(\sum_{j=i+1}^mA_j(x_j^{k+1}-x_j^k)-\sum_{j=i+1}^m A_j(x_j^{k}-x_j^{k-1})\right)\nonumber\\
&&+\beta\sum_{i=3}^m\sum_{j=2}^{i-1}(A_i x_i^{k+1}-A_i x_i)^\top (A_j x_j^{k+1}-A_j x_j^k), \;\forall\; w\in{\cal W}.\nonumber\\
\end{eqnarray}
Finally,  the assertion (\ref{convsol}) follows from  Lemma \ref{F-monotone} and inequality (\ref{lem33f}) immediately.
\end{proof}

\medskip
For succinctness, we temporarily skip the superscripts and the variables for $\Upsilon_i$ ($i=1,2$).
The next lemma focuses on analyzing the crossing terms
$\Upsilon_1$ and $\Upsilon_2$
in the right-hand side of (\ref{convsol}); and finding their lower bounds representable by negative quadratic terms.
The purpose of doing so is that the difference between the iterate $w^{k+1}$ and a solution point in ${\cal W}^*$ can be
completely represented by quadratic terms in a unified way. More specially, we decompose $\Upsilon_1$ and $\Upsilon_2$ into following several terms:
\begin{align}
& \Upsilon_1^{(1)}:=-\beta (A_2 x_2^{k+1}-A_2 x_2^k)^\top\left (\sum_{j=3}^m A_j (x_j^{k+1}-x_j^k)-\sum_{j=3}^m A_j (x_j^{k}-x_j^{k-1})\right)  \label{eq:ups11} \\
& \Upsilon_1^{(21)}:=-\beta\sum_{i=3}^{m-1} (A_i x_i^{k+1}-A_i x_i^k)^\top \sum_{j=i+1}^m A_j (x_j^{k+1}-x_j^k) \label{eq:ups121}\\
& \Upsilon_1^{(22)}:=\beta\sum_{i=3}^{m-1} (A_i x_i^{k+1}-A_i x_i^k)^\top \sum_{j=i+1}^m A_j (x_j^{k}-x_j^{k-1})\label{eq:ups122}\\
& \Upsilon_2^{(1)}:=\beta (A_2 x_2^{k+1}-A_2 x_2^k)^\top \sum_{i=3}^m(A_i x_i^{k+1}-A_i x_i)  \label{eq:ups21} \\
& \Upsilon_2^{(2)}:=\beta\sum_{j=3}^{m-1}\sum_{i=j+1}^m(A_i x_i^{k+1}-A_i x_i)^\top (A_j x_j^{k+1}-A_j x_j^k). \label{eq:ups122}
\end{align}
Then, we take a further analysis for each smaller term to get their lower bounds.

\medskip

\begin{lemma}
Suppose Assumptions \ref{AS1} and \ref{AS3} hold. Let $\{w^k\}$ be the sequence generated by the e-ADMM \eqref{EADMM}.
Then, for any $w\in {\cal W}$, we have the following assertions:

\begin{itemize}

\item [1)] For any scalars $a,b>0$, it holds
\begin{eqnarray}\label{Lemma341}
\Upsilon_1^{(1)}\ge \beta\left(-(a+b)\|A_2 x_2^k - A_2 x_2^{k+1}\|^2-\frac{1}{4a}\|\sum_{i=3}^m(A_i x_i^{k+1}-A_i x_i^k)\|^2-\frac{m-2}{4b}\sum_{i=3}^m\|A_i x_i^{k}-A_i x_i^{k-1})\|^2\right).\nonumber\\
\end{eqnarray}

\item [2)] The following identity holds:
\noindent\begin{eqnarray}\label{Lemma342}
\Upsilon_1^{(21)}=-\frac{\beta}{2}\left( \|\sum_{i=3}^m (A_i x_i^{k+1}-A_i x_i^k)\|^2-\sum_{i=3}^m\|A_i x_i^{k+1}-A_i x_i^k\|^2\right).
\end{eqnarray}
\item [3)] It holds that
\begin{eqnarray}
\label{Lemma343}
\Upsilon_1^{(22)}\ge-\frac{\beta}{2}\sum_{i=3}^{m-1}(m-i)\|A_i x_i^{k+1}- A_i x_i^k\|^2
-\frac{\beta}{2}\sum_{i=4}^m (i-3)\|A_i x_i^k-A_i x_i^{k-1}\|^2.
\end{eqnarray}
\item [4)] For any scalar $\delta>0$, it holds
\begin{eqnarray}\label{Lemma344}
&&\Upsilon_2^{(1)}\ge-\beta\left(\frac{m-2}{2\delta}\sum_{i=3}^m\|A_i x_i^{k+1}-A_i x_i\|^2
+\frac{\delta}{2}\|A_2 x_2^{k+1}-A_2 x_2^k\|^2\right).
\end{eqnarray}
\item [5)] It holds that
\begin{eqnarray}\label{Lemma345}
&&\Upsilon_2^{(2)}\ge-\frac{\beta}{2}\left(\sum_{i=3}^m(i-3)\|A_i x_i^{k+1}-A_i x_i\|^2+\sum_{i=3}^{m-1}(m-i)\|A_i x_i^{k+1}-A_i x_i^k\|^2\right).
\end{eqnarray}

\item [6)] $\Upsilon_1$ and $\Upsilon_2$ defined respectively in (\ref{Up1}) and (\ref{Up2}) satisfy the equations
\begin{eqnarray}
\label{UP1sum}
\Upsilon_1=\Upsilon_1^{(1)}+\Upsilon_1^{(21)}+\Upsilon_1^{(22)} \;\mbox{and}\;\Upsilon_2=\Upsilon_2^{(1)}+\Upsilon_2^{(2)}.
\end{eqnarray}
%
\end{itemize}
\end{lemma}

\begin{proof} 1) Using Cauchy-Schwarz inequality, for any positive scalars $a$ and $b$, we have
\begin{eqnarray*}\label{t1}
\Upsilon_1^{(1)}&=&-\beta (A_2 x_2^{k+1}-A_2 x_2^k)^\top \sum_{j=3}^m A_j (x_j^{k+1}-x_j^k)+\beta (A_2 x_2^{k+1}-A_2 x_2^k)^\top \sum_{j=3}^m A_j (x_j^{k}-x_j^{k-1})\nonumber\\
&\ge& \beta\left(-a\|A_2 x_2^k - A_2 x_2^{k+1}\|^2-\frac{1}{4a}\|\sum_{i=3}^m( A_i x_i^{k+1}-A_i x_i^k)\|^2\right)\nonumber\\
&&+\beta\left(-b\|A_2 x_2^k - A_2 x_2^{k+1}\|^2-\frac{m-2}{4b}\sum_{i=3}^m\|A_i x_i^{k}-A_i x_i^{k-1}\|^2\right).
\end{eqnarray*}
 Then, the inequality (\ref{Lemma341}) follows directly.

2)
Invoking the identity $x^\top y=\frac{1}{2}(\|x+y\|^2-\|x\|^2-\|y\|^2)$, we know
\begin{eqnarray*}\label{t2}
\Upsilon_1^{(21)}&=&-\frac{\beta}{2} \sum_{i=3}^{m-1}\left( \left\|\sum_{j=i}^m(A_j x_j^{k+1}-A_j x_j^k)\right\|^2-\left\|A_i x_i^{k+1}-A_i x_i^k\right\|^2 - \left\|\sum_{j=i+1}^m (A_j x_j^{k+1}-A_j x_j^k)\right\|^2\right)\nonumber\\
&=&-\frac{\beta}{2}\left( \|\sum_{i=3}^m (A_i x_i^{k+1}-A_i x_i^k)\|^2-\sum_{i=3}^m\|A_i x_i^{k+1}-A_i x_i^k\|^2\right).
\end{eqnarray*}
Then, the inequality (\ref{Lemma342}) is proved.

3) Using Cauchy-Schwarz inequality, we have
\begin{eqnarray*}
\label{t3}
\Upsilon_1^{(22)}&\ge&-\frac{\beta}{2}\sum_{i=3}^{m-1} \sum_{j=i+1}^m\left( \|A_i x_i^{k+1}- A_i x_i^k\|^2+\|A_j x_j^k-A_j x_j^{k-1}\|^2\right)\nonumber\\
&=&-\frac{\beta}{2}\sum_{i=3}^{m-1}(m-i)\|A_i x_i^{k+1}- A_i x_i^k\|^2-\frac{\beta}{2}\sum_{i=4}^m (i-3)\|A_i x_i^k-A_i x_i^{k-1}\|^2.
\end{eqnarray*}
Thus, we obtain the inequality (\ref{Lemma343}).

4) Using Cauchy-Schwarz inequality, for any positive scalar $\delta$, we have
\begin{eqnarray*}\label{t4}
\Upsilon_2^{(1)}&\ge&-\beta\left(\frac{1}{2\delta}\|\sum_{i=3}^m( A_i x_i^{k+1}-A_i x_i)\|^2 +\frac{\delta}{2}\|A_2 x_2^{k+1}-A_2 x_2^k\|^2\right)\nonumber\\
&\ge&-\beta\left(\frac{m-2}{2\delta}\sum_{i=3}^m\|A_i x_i^{k+1}-A_i x_i\|^2+\frac{\delta}{2}\|A_2 x_2^{k+1}-A_2 x_2^k\|^2\right).
\end{eqnarray*}
Then, the inequality (\ref{Lemma344}) follows directly.

5) Again, using Cauchy-Schwarz inequality, it yields
\begin{eqnarray*}\label{t5}
\Upsilon_2^{(2)} &\ge& -\frac{\beta}{2}\sum_{j=3}^{m-1}\sum_{i=j+1}^m\left(\|A_j x_j^{k+1} - A_j x_j^k\|^2+\|A_i x_i^{k+1}-A_i x_i\|^2 \right)\nonumber\\
&=&-\frac{\beta}{2}\left(\sum_{i=3}^m(i-3)\|A_i x_i^{k+1}-A_i x_i\|^2+\sum_{i=3}^{m-1}(m-i)\|A_i x_i^{k+1}-A_i x_i^k\|^2\right).
\end{eqnarray*}
 Thus, the inequality (\ref{Lemma345}) is proved.

6) The assertion (\ref{UP1sum}) follows from the definitions of $\Upsilon_1$, $\Upsilon_2$, $\Upsilon_1^{(1)}$, $\Upsilon_1^{(21)}$,
 $\Upsilon_1^{(22)}$, $\Upsilon_2^{(1)}$ and $\Upsilon_2^{(2)}$ (see (\ref{Up1}), (\ref{Up2}) and (\ref{eq:ups11}-\ref{eq:ups122})), and some elementary calculations.
\end{proof}

With the previously proved lemmas, we can  derive
a favorable relationship for two consecutive iterates about their respective differences from a solution point in ${\cal W}^*$. This relationship is reflected by an inequality that is completely representable by quadratic terms without any crossing terms. It is thus easy to show that the sequence generated by the e-ADMM (\ref{EADMM}) is Fej\`{e}r monotone with respect to ${\cal W}^*$.

\begin{lemma} \label{lem36}
Suppose Assumptions \ref{AS1} and \ref{AS3} hold. Let $\{w^k\}$ be the sequence generated by the e-ADMM \eqref{EADMM}. For arbitrary positive scalars $a$, $b$, $\delta$, and any $w^*\in {\cal W}^*$, we have
\begin{eqnarray}
\label{convtrasol}
&&\frac{\beta}{2}\sum_{i=2}^m\|A_i x_i^{k+1}-A_i x^*_i\|^2+\frac{1}{2\beta}\|z^{k+1}-z^*\|^2+\beta\sum_{i=3}^m\left[\frac{(i-3)}{2}+\frac{m-2}{4b}\right]\|A_i x_i^{k+1}-A_i x_i^k\|^2\nonumber\\
&&\le \frac{\beta}{2}\sum_{i=2}^m\|A_i x_i^{k}-A_i x^*_i\|^2+\frac{1}{2\beta}\|z^{k}-z^*\|^2+\beta\sum_{i=3}^m\left[\frac{(i-3)}{2}+\frac{m-2}{4b}\right]\|A_i x_i^{k}-A_i x_i^{k-1}\|^2\nonumber\\
\quad &&-\sum_{i=2}^m C_i\|A_i x_i^{k+1}-A_i x_i^k\|^2-\frac{1}{2\beta}\|z^k-z^{k+1}\|^2-\sum_{i=3}^m \zeta_i\|A_i x_i^{k+1}-A_i x_i^*\|^2,\nonumber\\
\end{eqnarray}
where
\begin{align}
& C_2:=(\frac{1}{2}-(a+b)-\frac{\delta}{2})\beta,\qquad\qquad\qquad  \label{rho} \\
& C_i:=\frac{\mu_i}{\|A_i^\top A_i\|}-\beta\left((\frac{1}{4a}+\frac{1}{4b})(m-2)+\frac{3m-i-7}{2}\right),\; i=3,\ldots,m,  \label{Ci}\\
&\mbox{and}\nonumber\\
&\zeta_i:=\frac{\mu_i}{\|A_i^\top A_i\|}-\beta\left[ \frac{m-2}{2\delta}+\frac{(i-3)}{2}\right], \;i=3,\ldots,m.\label{zeta2}
\end{align}
\end{lemma}

\begin{proof} First, substituting (\ref{Lemma341})-(\ref{Lemma345}) into (\ref{convsol}) and invoking (\ref{UP1sum}), we derive that
\begin{eqnarray}
\label{convtra}
&&\frac{\beta}{2}\sum_{i=2}^m\|A_i x_i^{k+1}-A_i x^*_i\|^2+\frac{1}{2\beta}\|z^{k+1}-z^*\|^2+\beta\sum_{i=3}^m\left[\frac{(i-3)}{2}+\frac{m-2}{4b}\right]\|A_i x_i^{k+1}-A_i x_i^k\|^2\nonumber\\
&&\le \frac{\beta}{2}\sum_{i=2}^m\|A_i x_i^{k}-A_i x^*_i\|^2+\frac{1}{2\beta}\|z^{k}-z^*\|^2+\beta\sum_{i=3}^m\left[\frac{(i-3)}{2}+\frac{m-2}{4b}\right]\|A_i x_i^{k}-A_i x_i^{k-1}\|^2\nonumber\\
&&-C_2\|A_2 x_2^{k+1}-A_2 x_2^k\|^2-\frac{1}{2\beta}\|z^k-z^{k+1}\|^2+\frac{\beta}{2}(1+\frac{1}{2a})\| \sum_{i=3}^m(A_i x_i^{k+1}-A_i x_i^k)\|^2\nonumber\\
&&-\sum_{i=3}^m\left\{ \mu_i\|x_i^k-x_i^{k+1}\|^2-\frac{\beta}{2}\left[ \frac{m-2}{2b}+2m-i-5 \right]\|A_ix_i^k-A_i x_i^{k+1}\|^2\right\}\nonumber\\
&&-\sum_{i=3}^m \left\{\frac{\mu_i}{2}\|x_i-x_i^{k+1}\|^2-\beta\left[ \frac{m-2}{2\delta}+\frac{(i-3)}{2}\right]\|A_i x_i^{k+1}-A_i x_i\|^2\right\}\nonumber\\
&&+\left\{ \theta(u)-\theta(u^{k+1})+(w-w^{k+1})^\top F(w)+\beta(\sum_{i=1}^m A_i x_i -b)^\top \sum_{i=1}^m A_i (x_i^k-x_i^{k+1})\right\}.\nonumber\\
\end{eqnarray}
Invoking the Cauchy-Schwarz inequality and $a>0$, we have
\begin{eqnarray}\label{abm}
\frac{\beta}{2}(1+\frac{1}{2a})\| \sum_{i=3}^m(A_i x_i^{k+1}-A_i x_i^k)\|^2\le\frac{\beta}{2}(1+\frac{1}{2a})(m-2)\sum_{i=3}^m\|A_i x_i^{k+1}-A_i x_i^k\|^2.
\end{eqnarray}
Then, using (\ref{MSVI-F}), we get
\begin{eqnarray}
\label{solst}
 \theta(u^*)-\theta(u^{k+1})+(w^*-w^{k+1})^\top F(w^*)+\beta(\sum_{i=1}^m A_i x_i^* -b)^\top \sum_{i=1}^m A_i (x_i^k-x_i^{k+1})\le -\sum_{i=3}^m\frac{\mu_i}{2}\|x_i^*-x_i^{k+1}\|^2.\nonumber\\
 \end{eqnarray}
Setting $w:=w^*\in {\cal W}^*$ in (\ref{convtra}) and combining (\ref{abm}) and (\ref{solst}),  we obtain the assertion (\ref{convtrasol})-(\ref{zeta2}) directly.
\end{proof}

\subsection{Main result}
In this subsection, we prove the convergence of the e-ADMM \eqref{EADMM} with $m\ge3$ under  Assumptions \ref{AS1} and \ref{AS3}. This is the main result of this paper. As mentioned, it has been shown in \cite{CHYY} that the penalty parameter $\beta$ must be appropriately restricted to guarantee the convergence of the e-ADMM \eqref{EADMM} even all functions are assumed to be strongly convex. Therefore, in the following theorem we first present a range of $\beta$ to ensure the convergence of the e-ADMM \eqref{EADMM} with $m\ge 3$ under our assumptions. We target a larger range of $\beta$ while ensuring that all the coefficients $C_i$ ($i=2,\ldots,m$) and $\zeta_i$ ($i=3,\ldots,m$) defined in (\ref{rho})-(\ref{zeta2}) are positive. With the positiveness of these coefficients, as we shall show in the proof, it becomes possible to measure the difference of distance to a solution point for two consecutive iterates. It is noticed that determining the range of $\beta$ via the inequalities (\ref{rho})-(\ref{zeta2}) relies on the free variables $a,b,\delta$ and $m$; thus it seems to be unclear to know what the theoretically largest range is. In the following proof, we provide a heuristics and probe a favorable range of $\beta$ which can be shown to be a better choice than those in the existing literature.

\begin{lemma}\label{betacoeff}
Suppose Assumptions \ref{AS1} and \ref{AS3} hold. When $\beta$ is restricted by
\begin{eqnarray}
\label{betrangup}
\beta \in \left(0,\;
\min_{3\le i\le m} \frac{\mu_i}{\max\{4m-10, 3m-6.5\}\|A_i^\top A_i\|}\right),
 \end{eqnarray}
we have
  \begin{itemize}
  \item[i)]  $C_i>0$, $i=2,\ldots,m$,
  \item[ii)]  $\zeta_i>0$, $i=3,\ldots,m$,
  \end{itemize}
where $C_2$, $C_i$ ($i=3,\cdots,m$) and $\zeta_i$ ($i=3,\cdots,m$) are defined in (\ref{rho}), (\ref{Ci}) and (\ref{zeta2}), respectively.
\end{lemma}
\begin{proof}
Let us first explain our heuristics to find the range (\ref{betrangup}). With the purpose of finding a larger range of $\beta$ while enduring the positiveness of all the coefficients in (\ref{rho})-(\ref{zeta2}), and mainly motivated by (\ref{Ci}), we choose $a=b$ and thus we should ensure the following inequalities:
\begin{align}
& C_2:=(\frac{1}{2}-2a-\frac{\delta}{2})\beta >0,\qquad\qquad\qquad  \label{rho-2} \\
& C_i:=\frac{\mu_i}{\|A_i^\top A_i\|}-\beta\left(\frac{m-2}{2a}+\frac{3m-i-7}{2}\right)>0,\; i=3,\ldots,m,  \label{Ci-2}\\
&\zeta_i:=\frac{\mu_i}{\|A_i^\top A_i\|}-\beta\left[ \frac{m-2}{2\delta}+\frac{(i-3)}{2}\right]>0, \;i=3,\ldots,m.\label{zeta-2}
\end{align}
To ensure (\ref{rho-2}) and simplify (\ref{zeta-2}), we probe the choice of $\delta$ as
$$
\delta=(1-4a)\frac{m-2}{m-(2-\epsilon')} \;\;\hbox{with}\; \epsilon'>0
$$
so that the numerator $m-2$ in (\ref{zeta-2}) can be canceled. This particular choice also makes us to derive a range of $\beta$ whose upper bound can be represented by some linear terms of $m$. Indeed, with the mentioned probe, we have
\begin{align}
& C_2:=\frac{1-4a}{2}\frac{\epsilon'}{m-2+\epsilon'},\qquad\qquad\qquad  \label{rho-3} \\
& C_i:=\frac{\mu_i}{\|A_i^\top A_i\|}-\beta\left(\frac{m-2}{2a}+\frac{3m-i-7}{2}\right),\; i=3,\ldots,m,  \label{Ci-3}\\
&\zeta_i:=\frac{\mu_i}{\|A_i^\top A_i\|}-\beta\left[ \frac{m-2+\epsilon'}{2(1-4a)}+\frac{i-3}{2}\right], \;i=3,\ldots,m.\label{zeta-3}
\end{align}
Further probing different values of $a$, we choose $a=\frac{1}{5}$ in (\ref{rho-3}). Also, we choose $i=3$ in (\ref{Ci-3}) and $i=m$ in (\ref{zeta-3}).
Then, the definitions in (\ref{rho-3})-(\ref{zeta-3}) can be accordingly specified as
\begin{align}
& C_2:=\frac{\epsilon'}{10(m-2+\epsilon')},\qquad\qquad\qquad  \label{rho-4} \\
& C_i:=\frac{\mu_i}{\|A_i^\top A_i\|}-\beta(4m-10),\; i=3,\ldots,m,  \label{Ci-4}\\
&\zeta_i:=\frac{\mu_i}{\|A_i^\top A_i\|}-\beta(3m-6.5+2.5\epsilon'), \;i=3,\ldots,m.\label{zeta-4}
\end{align}
Letting $\epsilon' \to 0$ in (\ref{zeta-4}), we obtain the range (\ref{betrangup}) that can ensure the positiveness of all the coefficients defined in (\ref{rho})-(\ref{zeta2}).
\end{proof}

Now, we are in the stage to prove the convergence of the e-ADMM \eqref{EADMM} with the restriction (\ref{betrangup}) on $\beta$.
Let us define a potential function $\Phi(v^{k+1},v^k,v)$ as
\begin{eqnarray}\label{Phik}
\Phi(v^{k+1},v^k,v):=\frac{1}{2}\| v^{k+1}-{v}\|_Q^2+\beta\sum_{i=3}^m\left[\frac{(i-3)}{2}+\frac{5(m-2)}{4}\right]\|A_i x_i^{k+1}-A_i x_i^k\|^2,
\end{eqnarray}
with $\beta$ satisfying (\ref{betrangup}); and a block diagonal matrix as
\begin{eqnarray}\label{tildeQ}
\tilde Q=\hbox{diag}\left(2C_2 A_2^\top A_2, \cdots,2C_m A_m^\top A_m,\frac{1}{\beta}I\right).
\end{eqnarray}
Then, ${\tilde Q}$ is positive definite because $C_i>0$ ($i=2,\cdots,m$) according to Lemma \ref{betacoeff}. It thus follows from (\ref{convtrasol}) that
 \begin{eqnarray}\label{T1}
 \Phi(v^{k+1},v^k,v^*)\le\Phi(v^{k},v^{k-1},v^*)-\frac{1}{2}\|v^k-v^{k+1}\|^2_{\tilde Q}.
 \end{eqnarray}

\begin{theorem}\label{Thm-convergence}
Suppose Assumptions \ref{AS1} and \ref{AS3} hold. Let $\{w^k\}$ be the sequence generated by the e-ADMM \eqref{EADMM} with $\beta$ restricted in (\ref{betrangup}). Then, the sequence $\{w^k\}$ converges to a solution point in ${\cal W}^*$.
\end{theorem}

\begin{proof}
It follows from (\ref{T1}) that the sequence defined by
$$
\big\{\frac{1}{2}\| v^{k+1}-{v^*}\|_Q^2+\beta\sum_{i=3}^m\left[\frac{(i-3)}{2}+\frac{5(m-2)}{4}\right]\|A_i x_i^{k+1}-A_i x_i^k\|^2\big\}
$$
is non-increasing.
It implies that the sequence $\{ v^k\}$ is bounded under the assumption that $A_i$ $(i=2,\ldots,m)$ are full column rank. The relationship \eqref{lap} and the fact that $A_1$ is full column rank  further imply that the sequence $\{x_1^k\}$ is also bounded. Hence, the iterative sequence $\{w^k\}$ generated by the scheme \eqref{x1a}-\eqref{lap} is bounded.

Summarizing (\ref{T1}) for all $k$ and rearranging the terms, we get
\begin{eqnarray}\label{Global}
&&\frac{1}{2}\sum_{k=1}^{\infty}\|v^k-v^{k+1}\|^2_{\tilde Q}\le \left(\frac{1}{2}\| v^{1}-{ v^*}\|_Q^2 +\beta\sum_{i=3}^m\left[\frac{(i-3)}{2}+\frac{5(m-2)}{4}\right]\|A_i x_i^{0}-A_i x_i^{1}\|^2\right),
\end{eqnarray}
which implies
\[\label{lim}
\lim_{k\rightarrow\infty}\|z^k-z^{k+1}\|=0\;\mbox{and}\;\lim_{k\rightarrow\infty}\|A_i x_i^k-A_i x_i^{k+1}\|=0, \;\;i=2,\ldots,m.
\]
Moreover, the boundedness of the sequence indicates that the sequence $\{w^k\}$ has at least one cluster point. Let $w^{\infty}$ be an arbitrary cluster point of $\{w^k\} $ and $\{w^{k_j}\}$ be the subsequence converging to $w^{\infty}$. Then, the sequence $\{v^{k_j}\}$ converges to $v^{\infty}$; and the whole sequence $\{v^k\}$ has only one cluster point $v^{\infty}$ because of \eqref{T1}.
On the other hand, it follows from (\ref{lap}) and the fact that $A_1$ is full column rank that
\begin{eqnarray*}
\label{x1new}
x_1^{k+1}=(A_1^\top A_1)^{-1}A_1^\top\left(b-\sum_{i=2}^mA_i x_i^{k+1}+\frac{z^k-z^{k+1}}{\beta}\right).
\end{eqnarray*}
Then, the sequence $\{x_1^k\}$ has only one cluster point, say $x_1^\infty:=(A_1^\top A_1)^{-1}A_1^\top\left(b-\sum_{i=2}^mA_i x_i^{\infty}\right)$, by combining the above equation with $v^k\rightarrow v^\infty$.
 Thus, the sequence $\{w^k\}$ converges to $w^\infty$. Taking limit along the subsequence $\{w^{k_j}\}$ in \eqref{optint} and using  \eqref{lim}, we have

\begin{eqnarray*} \label{MA-VIsolm}
    \left\{
    \begin{array}{l}
     \theta_1(x_1) - \theta_1(x_1^\infty) + (x_1- x_1^\infty)^\top (- A_1^\top z^\infty) \ge0, \\
      \theta_2(x_2) - \theta_2(x_2^\infty) + (x_2- x_2^\infty)^\top (- A_2^\top z^\infty) \ge 0, \\
          \quad  \cdots \cdots \quad\qquad \cdots \cdots\\
       \theta_i(x_i) - \theta_i(x_i^\infty) + (x_i- x_i^\infty)^\top (- A_i^\top z^\infty) \ge \frac{\mu_i}{2}\|x_i-x_i^\infty\|^2,\;\;i=3,\ldots,m,\nonumber \\
            \quad  \cdots \cdots \quad\qquad \cdots \cdots\\
     \sum_{i=1}^m A_ix_i^\infty - b =0.
      \end{array} \right.\;\forall\; w\in{\cal W}.
    \end{eqnarray*}
According to the optimality condition \eqref{MSVII}, we know $w^\infty \in {\cal W^*}$. Consequently, the sequence $\{w^k\}$ generated by the e-ADMM \eqref{EADMM} with $\beta$ restricted in (\ref{betrangup}) converges to a solution point in ${\cal W^*}$.
\end{proof}

\begin{remark}\label{dropfull}
It can be seen from (\ref{T1}) that the sequence $\{A_i x_i^k\}$ $(i=1,\ldots,m)$ converges to $\{A_i x_i^{\infty}\}$ $(i=1,\ldots,m)$ even without the full column rank assumptions on $A_i$'s $(i=1,\ldots,m)$.
%
\end{remark}
\begin{remark} We have shown the convergence of the e-ADMM \eqref{EADMM} when $\beta$ is restricted in the range (\ref{betrangup}), for the generic case with a general $m\ge3$. Indeed, when the special case of $m=3$ is considered, the range (\ref{betrangup}) reduces to
\begin{eqnarray}\nonumber
\beta\in \left(0, \frac{2}{5 \|A_3^\top A_3\|}\right),
\end{eqnarray}
which is still larger than some ones in the literature that are only eligible to the special case of $m=3$, e.g., the range $\left(0, \frac{6}{17\|A_3^\top A_3\|}\right)$ proposed in \cite{CHY}.
\end{remark}

\section{Ergodic convergence rate}
\label{EgoCov}
In \cite{HeYuan-SINUM,HeYuan-NM}, some worst-case $O(1/t)$ convergence rates measured by the iteration complexity were established for the original ADMM scheme which corresponds to the scheme \eqref{EADMM} with $m=2$. Since then, there are some works focusing on investigating the convergence rates in the same nature for various splitting methods in the literature. This kind of convergence rates provides a global, though perhaps not sharp, estimate on the convergence speed for the algorithm under discussion. In this section, we establish a worst-case $O(1/t)$ convergence rate measured by iteration complexity for the e-ADMM (\ref{EADMM}) with $m \ge 3$ under Assumptions \ref{AS1} and \ref{AS3}. Compared with (\ref{betrangup}), the restriction on $\beta$ to ensure the $O(1/t)$ convergence rate is slightly more restrictive. In order to establish the ergodic convergence rate, we require the positiveness of $C_2$ defined in (\ref{rho}), $C_i$ $(i=3,\ldots,m)$ in (\ref{Ci}) and $\tilde\zeta_i$ ($i=3,\ldots,m$) in (\ref{zeta3}). Note that $\tilde\zeta_i$ ($i=3,\ldots,m$) is deferent from $\zeta_i$ ($i=3,\ldots,m$); and the difference results in a more restrictive range of $\beta$ as to be shown later.

We first prove a lemma that will be used to prove a worst-case $O(1/t)$ convergence rate for the e-ADMM (\ref{EADMM}) with $m \ge 3$.
\begin{lemma}\label{lem41}
Suppose Assumptions \ref{AS1} and \ref{AS3} hold. Let $\{w^k\}$ be the sequence generated by the e-ADMM \eqref{EADMM} with $m\ge3$. If $\beta$ is restricted by
\begin{eqnarray}
\label{betrang2}
\beta \in \left(0,
\min_{3\le i\le m} \frac{\mu_i}{(\frac{13+\sqrt{33}}{4}m-\frac{17+\sqrt{33}}{2})\|A_i^\top A_i\|}\right),
 \end{eqnarray}
then we have
\begin{eqnarray}\label{Erogcont}
 \Theta(v^{k+1}, v^k,v)\le \Theta(v^{k}, v^{k-1},v)+\Xi(w^{k+1},w^k,w),
\end{eqnarray}
where
\begin{align}
&\Theta(v^{k+1}, v^k,v):=\frac{1}{2}\|v^{k+1}-v\|_Q^2+\beta\sum_{i=3}^m\tau_i\|A_i x_i^{k+1}-A_i x_i^{k}\|^2, \;\label{The}\\
&\tau_i:=\frac{(i-3)}{2}+\frac{(7+\sqrt{33})(m-2)}{8}.\label{taui}
\end{align}
 and
 \begin{eqnarray}\Xi(w^{k+1},w^k,w):=\theta(u)-\theta(u^{k+1})+(w-w^{k+1})^\top F(w)+\beta(\sum_{i=1}^m A_i x_i -b)^\top \sum_{i=2}^m A_i (x_i^k-x_i^{k+1}). \label{LXi}
\end{eqnarray}
\end{lemma}

\begin{proof} First, substituting (\ref{abm}) into  (\ref{convtra}), we get
  \begin{eqnarray}
\label{ergokey}
&&\frac{\beta}{2}\sum_{i=2}^m\|A_i x_i^{k+1}-A_i x_i\|^2+\frac{1}{2\beta}\|z^{k+1}-z\|^2+\beta\sum_{i=3}^m\left\{\frac{(i-3)}{2}+\frac{(m-2)}{4b}\right\}\|A_i x_i^{k+1}-A_i x_i^k\|^2\nonumber\\
&&\le \frac{\beta}{2}\sum_{i=2}^m\|A_i x_i^{k}-A_i x_i\|^2+\frac{1}{2\beta}\|z^{k}-z\|^2+\beta\sum_{i=3}^m\left\{\frac{(i-3)}{2}+\frac{(m-2)}{4b}\right\}\|A_i x_i^{k}-A_i x_i^{k-1}\|^2\nonumber\\
&&-\sum_{i=2}^m C_i\|A_i x_i^{k+1}-A_i x_i^k\|^2-\frac{1}{2\beta}\|z^k-z^{k+1}\|^2-\sum_{i=3}^m \tilde\zeta_i\|A_i x_i^{k+1}-A_i x_i\|^2+\Xi(w^{k+1},w^k,w),\nonumber\\
\end{eqnarray}
where $\Xi(w^{k+1},w^k,w)$ is defined in (\ref{LXi}) and
\begin{align}
& \tilde\zeta_i:=\frac{\mu_i}{2\|A_i^\top A_i\|}-\beta\left[ \frac{m-2}{2\delta}+\frac{(i-3)}{2}\right], \;i=3,\ldots,m. \label{zeta3} \\
\nonumber
\end{align}
The heuristics of the following part is similar as that of Lemma \ref{betacoeff}. We skip the detail for succinctness.
Setting $a=b=\frac{7-\sqrt{33}}{8}$ and $\delta=\frac{\sqrt{33}-5}{2}\frac{(m-2)}{(m-2+\epsilon')}$ $(\epsilon'>0)$,
we get
$$C_2=\frac{\sqrt{33}-5}{4}\frac{\epsilon'}{(m-2+\epsilon')},\qquad\qquad\qquad\qquad\qquad\qquad$$
$$ C_i=\frac{\mu_i}{\|A_i^\top A_i\|}-(\frac{13+\sqrt{33}}{4}m-\frac{14+i+\sqrt{33}}{2})\beta,\;i=3,\ldots,m,\qquad\qquad$$
and
$$\tilde \zeta_i=\frac{\mu_i}{2\|A_i^\top A_i\|}-(
\frac{5+\sqrt{33}}{8}m-\frac{5+\sqrt{33}}{4}+\frac{i-3}{2}+\frac{5+\sqrt{33}}{8}\epsilon')\beta,\;i=3,\ldots,m. $$
Let $\epsilon'\rightarrow 0+$. Then, we derive that $C_i>0$, $(i=2,\ldots,m)$ and $\tilde\zeta_i>0$ ($i=3,\ldots,m$)
when $\beta$ satisfies (\ref{betrang2}). Thus, the assertion (\ref{Erogcont}) follows from (\ref{ergokey}) immediately.
\end{proof}

Based on Lemma \ref{lem41}, we now establish a worst-case $O(1/t)$ convergence rate in the ergodic sense for the e-ADMM \eqref{x1a}-\eqref{lap}. For this analysis, the quality of an iterate is measured by the feasibility violation and the decrease of the objective function. Let us define
\[\label{XK}
x_{i,t}^{k+1}:=\frac{1}{t}\sum_{k=1}^t  x_i^{k+1},\;i=1,\ldots,m; \;\;\; \;u^{k+1}_t:=\frac{1}{t}\sum_{k=1}^t u^{k+1},\; \;\hbox{and} \;\; w^{k+1}_t:=\frac{1}{t}\sum_{k=1}^t w^{k+1}.\]
Obviously, $w^{k+1}_t \in{\cal W}$ because of the convexity of ${\cal X}_i$ $(i=1,\ldots,m)$. Note that we are considering the case of $m\ge3$. Hence, the interval (\ref{betrang2}) is included in the restriction of $\beta$ (\ref{betrangup}). Then, invoking Theorem \ref{Thm-convergence}, the sequence $\{\frac{1}{2}\|v^k-v^*\|_Q^2\}$ is bounded and thus there exists a constant $\kappa$ such that
\begin{eqnarray}\label{kappab}
\|A_ix_i^k\|\le \kappa, \;\forall\; i=1, \ldots,m,\;\hbox{and}\; \|z^k\|\le \kappa, \;\forall \;k\ge 0.
\end{eqnarray}

\begin{theorem}\label{TA2}
Suppose Assumptions \ref{AS1} and \ref{AS3} hold. For $t$ iterations generated by the e-ADMM \eqref{EADMM} with $\beta$ restricted in (\ref{betrang2}), the following assertions holds.
\begin{itemize}
\item[1)] For $\bar C:=\beta\sum_{i=3}^m\tau_i\|A_i x_i^1-A_i x_i^0\|^2$ and $\tau_i$ is defined in (\ref{taui}), we have
\begin{eqnarray}\label{BK}
 \theta(u_t^{k+1})-\theta(u)+( w_t^{k+1}-w)^\top F(w)\le \frac{1}{t}\left[2\beta \kappa(m-1)\|\sum_{i=1}^m A_i x_i-b\|
  +\frac{1}{2}\|v^1-v\|_Q^2+{\bar C}\right].\quad
\end{eqnarray}

\item[2)] There exists a constant ${\bar c}_1>0$ such that
\begin{eqnarray}\label{primalergo}
\|\sum_{i=1}^m A_i x_{i,t}^{k+1}-b\|^2\le \frac{{\bar c}_1}{t^2}.
\end{eqnarray}
\item[3)] There exists a constant ${\bar c}_2>0$ such that
\begin{eqnarray}
\label{objergo}
|\theta(u_t^{k+1})-\theta(u^*)|\le\frac{{\bar c}_2}{t}.
\end{eqnarray}
\end{itemize}
\end{theorem}
\begin{proof}
1) First,
it follows from  the assertion \eqref{Erogcont} that
for all $w\in{\cal W}$, we have
\begin{eqnarray}\label{Lem2p}
&&\theta(u)-\theta(u^{k+1})+(w-w^{k+1})^\top F(w)+\beta(\sum_{i=1}^m A_i x_i-b)^\top \sum_{i=2}^m A_i(x_i^k-x_i^{k+1})\nonumber\\
&&\ge \Theta(v^{k+1}, v^k,v )
-\Theta(v^{k}, v^{k-1},v ).
\end{eqnarray}
Summarizing both sides of the above inequality from $k=1, 2, \cdots, t$, we have
\begin{eqnarray}\label{sum}
&&t\theta(u)-\sum_{k=1}^t\theta(u^{k+1})+(tw-\sum_{k=1}^t w^{k+1})^\top F(w)+\beta (\sum_{i=1}^mA_i x_i-b)^\top \sum_{i=2}^m A_i (x_i^1-x_i^{t+1})\nonumber\\
&&\ge \Theta(v^{t+1}, v^t, v )-\Theta(v^{1}, v^0,v ).
\end{eqnarray}
Then, it follows from the convexity of $\theta_i$ $(i=1,\ldots,m)$ that
\[\label{FK}
\theta(u_t^{k+1})\le\frac{1}{t}\sum_{k=1}^t\theta(u^{k+1}).
\]
Combining (\ref{kappab}), (\ref{sum}) and (\ref{FK}), we have
\begin{eqnarray}\label{add1}
\theta(u_t^{k+1})-\theta(u)+( w_t^{k+1}-w)^\top F(w)\le\frac{1}{t} \left(\Theta(v^{1}, v^0,v )+2\beta\kappa(m-1)\|\sum_{i=1}^m A_i x_i-b\|\right).
\end{eqnarray}
Thus, the assertion (\ref{BK}) follows from the above inequality and the defintion of $\Theta(v^{1}, v^0,v)$ directly.

2) Let us define ${\bar c}_1:=\frac{2}{\beta^2}\left(\|z^1-z^*\|^2+\|z^{k+1}-z^*\|^2\right)$. Then, we have
\begin{eqnarray}\nonumber
\|\sum_{i=1}^m A_i x_{i,t}^{k+1}-b\|^2 &=&\left\|\frac{1}{t}\sum_{k=1}^t\left[ \sum_{i=1}^m A_i x_{i}^{k+1}-b\right]\right\|^2 =\left\|\frac{1}{t}\sum_{k=1}^t\left[  \frac{1}{\beta}(z^k-z^{k+1})\right]\right\|^2=\left\|\frac{1}{t} \frac{1}{\beta}\left( z^1-z^{t+1}\right)\right\|^2\nonumber\\
&\le&\frac{1}{t^2}\frac{2}{\beta^2}\left(\|z^1-z^*\|^2+\|z^{k+1}-z^*\|^2\right)=\frac{{\bar c_1}}{t^2},\nonumber
\end{eqnarray}
where the first equality follows from (\ref{XK}), the second follows from (\ref{lap}) and the last follows from Cauchy-Schwarz inequality.
The assertion (\ref{primalergo}) is proved immediately.

3) It follows from $L(u_t^{k+1},z^*)\ge L(u^*,z^*)$ that
\begin{eqnarray}
\label{objlowerg}
\theta(u_t^{k+1})-\theta(u^*)\ge\langle z^*,\sum_{i=1}^m A_i x_{i,t}^{k+1}-b\rangle\ge-\frac{1}{2}\left(\frac{1}{t}\|z^*\|^2+t\|\sum_{i=1}^m A_i x_{i,t}^{k+1}-b\|^2 \right)\ge-\frac{1}{2t}(\|z^*\|^2+{\bar c}_1),\nonumber\\
\end{eqnarray}
where the second inequality is because of the Cauchy-Schwarz inequality and the last is due to (\ref{primalergo}).
On the other hand, setting $w:=w^*$ in (\ref{add1}), we obtain
\begin{eqnarray}\nonumber
&& \theta(u_t^{k+1})-\theta(u^*)+( w_t^{k+1}-w^*)^\top F(w^*)\le\frac{1}{t}\Theta(v^{1}, v^0,v^*).
\end{eqnarray}
Invoking the definition of $F$ in (\ref{MD-F}), we have
\begin{eqnarray}\nonumber
( w_t^{k+1}-w^*)^\top F(w^*)=-\langle z^*,\sum_{i=1}^m A_i x_{i,t}^{k+1}-b \rangle\ge-\frac{1}{2t}(\|z^*\|^2+{\bar c}_1),
\end{eqnarray}
where the proof of the last inequality is similar to (\ref{objlowerg}).
Combining these two inequalities above, we get
\begin{eqnarray}\label{objuperg}
\theta(u_t^{k+1})-\theta(u^*)\le\frac{1}{t}\Theta(v^{1}, v^0,v^*)+\frac{1}{2t}(\|z^*\|^2+{\bar c}_1).
\end{eqnarray}
The inequalities (\ref{objlowerg}) and (\ref{objuperg}) indicate that the assertion (\ref{objergo}) holds by setting $\bar c_2:=\Theta(v^{1}, v^0,v^*)+\frac{1}{2}(\|z^*\|^2+{\bar c}_1)$.

\end{proof}

For a compact set $D\subset{\cal W}$ containing a solution point of the variational inequality (\ref{MSVII}), let us define
\begin{equation}\label{Viapp}
\tilde d:=\sup\{2\beta \kappa(m-1)\|\sum_{i=1}^m A_i x_i-b\|
  +\frac{1}{2}\|v^1-v\|_Q^2|w\in D  \}.
  \end{equation}
Then, for the first $t$ iterations of the e-ADMM (\ref{EADMM}), the point $w^{k+1}_t$ defined in (\ref{XK}) satisfies
\begin{eqnarray}\label{VIapp}
\sup_{w\in {\cal D}}\{\theta(u_t^{k+1})-\theta(u)+( w_t^{k+1}-w)^\top F(w) \}\le\frac{\tilde d +\bar C}{t}.
\end{eqnarray}
On the other hand, invoking Theorem \ref{Thm-convergence} and  Stolz-Ces\`{a}ro Theorem (see, e.g. \cite{ST}), the sequence $\{w^{k+1}_t\}$ converges to the same saddle point $w^{\infty}$ as the sequence $\{w^k\}$ does. Therefore, it implies that $w_t^{k+1}$ is an approximated solution of (\ref{mini3}) with an accuracy of $O(1/t)$ in sense of (\ref{VIapp}). Note that Theorem \ref{TA2} also indicates a worst-case $O(1/t)$ convergence rate of the e-ADMM (\ref{EADMM}) in sense of that the accuracy is measured by both the feasibility violation (\ref{primalergo}) and the decrease of the objective function (\ref{objergo}).

\begin{remark}
For the special case of $m=3$, the restriction of $\beta$ (\ref{betrang2}) reduces to
\begin{eqnarray}\nonumber
 \beta\in(0,\frac{\sqrt{33}-5}{2\|A_3^\top A_3\|}),\end{eqnarray}
 which is larger than that in \cite{CHY} for the special case of $m=3$:
\begin{eqnarray}\nonumber
\beta\in(0,\frac{6}{17\|A_3^\top A_3\|}).\end{eqnarray}
For $m\ge4$, we can derive a less restrictive range for $\beta$:
\begin{eqnarray}
\label{ergomg4}
\beta \in \left(0,
\min_{3\le i\le m} \frac{\mu_i}{\max\{4.5m-11\}\|A_i^\top A_i\|}\right).
 \end{eqnarray}
Indeed, setting $a=b=\frac{1}{6}$ and $\delta=\frac{1}{3}\frac{(m-2)}{(m-2+\epsilon')}$ $(\epsilon'>0)$.
Let $\epsilon'\rightarrow 0+$,
we know that $C_i>0$, $(i=2,\ldots,m)$ and $\tilde\zeta_i>0$ ($i=3,\ldots,m$)
when $\beta$ satisfies (\ref{ergomg4}). Thus,
the inequality (\ref{Erogcont})
  holds with $\tau_i=\frac{(i-3)}{2}+\frac{3(m-2)}{2}$ when $\beta$ is restricted to (\ref{ergomg4}). Thus, Theorem \ref{TA2} also holds and it ensures a worst-case $O(1/t)$ ergodic convergence rate in senses of both the variational inequality characterization (\ref{BK}) and the combination of the feasibility violation (\ref{primalergo}) and the decrease of the objective function (\ref{objergo}).
\end{remark}

\section{Asymptotically linear convergence under stronger conditions}

In this section, we show that it is possible to theoretically derive the globally linear convergence in the asymptotical sense for the e-ADMM (\ref{EADMM}) with $m\ge 3$. The results in Section \ref{conv} are useful for this purpose. Note that the asymptotically linear convergence is a very strong result and thus  more general assumptions are needed to ensure this result. We refer to \cite{LMZb} for some existing results about the linear convergence of the e-ADMM (\ref{EADMM}) under some conditions stronger than what we shall present now. Our assumptions to ensure the asymptotically linear convergence of the e-ADMM (\ref{EADMM}) are listed below.

\begin{ass}\label{ASm}
In \eqref{mini3}, $\theta_1$ is convex and $\theta_i$ $(i=2,\ldots,m)$ are strongly convex with modulus $\mu_i$. Moreover, one of the following conditions hold:
\begin{itemize}
\item [1)]
 One of $\nabla\theta_i$ ($i=2,\ldots,m$) is Lipschitz continuous with constant $L_i$, the corresponding $A_i$ is full row rank and the corresponding $\X_i=\Re^{n_i}$;
\item [2)]
 $\nabla\theta_1$ is Lipschitz continuous with constant $L_1$,  $A_1$ is nonsingular and $\X_1=\Re^{n_1}$.
 \end{itemize}
\end{ass}

\medskip

First of all, under Assumption \ref{ASm}, we can prove a result similar to \eqref{convtrasol}, i.e.,
\begin{eqnarray}\label{onestr}
&& \Phi(v^{k+1},v^k,v^*)\le\Phi(v^{k},v^{k-1},v^*)-\sum_{i=2}^m C_i\|A_i x_i^{k+1}-A_i x_i^k\|^2-\frac{1}{2\beta}\|z^k-z^{k+1}\|^2\nonumber\\
&&-\sum_{i=3}^m\zeta_i\|A_i x_i^{k+1}-A_i x_i^*\|^2-\mu_2\|x_2^{k+1}-x_2^*\|^2-\mu_2\|x_2^{k+1}-x_2^k\|^2,\nonumber\\
\end{eqnarray}
where $\Phi(v^{k+1},v^k,v)$ is defined in (\ref{Phik}). Recall that when the penalty parameter $\beta$ is restricted into (\ref{betrangup}), we know that all the constants $C_2$ in (\ref{rho}), $C_i$ $(i=3,\ldots,m)$ in (\ref{Ci}) and $\zeta_i$ in (\ref{zeta2}) are  positive. Thus, there exists a constant
$$
\varsigma:=\min\left\{\min_{3\le i\le m}\{\zeta_i\}, \min_{2\le i\le m}\{C_i\},   \frac{1}{2}\max_{1\le i\le m}\frac{\max\{4m-10, 3m-6.5\}\|A_i^\top A_i\|}{\mu_i},\mu_2\right\}>0
$$
such that
\begin{eqnarray}\label{K1}
&&\Phi(v^{k+1},v^k,v^*)\le\Phi(v^{k},v^{k-1},v^*)\nonumber\\
&&-\varsigma\left(\sum_{i=2}^m \|A_i x_i^{k+1}-A_i x_i^k\|^2+\|z^k-z^{k+1}\|^2+\sum_{i=3}^m\|A_i x_i^{k+1}-A_i x_i^*\|^2 +\|x_2^{k+1}-x_2^*\|^2+\|x_2^{k+1}-x_2^k\|^2\right).\nonumber\\
\end{eqnarray}
Indeed, according to \eqref{K1} and the definition of $\Phi(v^{k+1},v^k,v)$ in (\ref{Phik}), it is clear that we only need to  bound the terms $\|z^{k+1}- z^*\|^2$ in terms of the minus term in \eqref{K1}.
As we show below, this is exactly why we need to assume Assumption \ref{ASm}.

\medskip

\begin{lemma}\label{lamB}
Suppose Assumption \ref{ASm} holds. Let $w^*$ be a saddle point in ${\cal W}^*$ and $\{w^k\}$ be the sequence generated by the e-ADMM \eqref{EADMM} with $m\ge 3$. Then, there exists a constant $\sigma_1>0$ such that
\begin{eqnarray}
\label{LB1}
\|z^{k+1}- z^*\|^2\le \sigma_1\left(\sum_{i=2}^m \|A_i x_i^{k+1}-A_i x_i^k\|^2+\sum_{i=3}^m\|A_i x_i^{k+1}-A_i x_i^*\|^2\right.\nonumber\\
\left.+\|z^k-z^{k+1}\|^2+\|x_2^{k+1}-x_2^*\|^2+\|x_2^{k+1}-x_2^k\|^2 \right).\end{eqnarray}
\end{lemma}

\begin{proof}
We consider two cases.

\indent{\bf Case I}: `` One of $\nabla\theta_i$ ($i=2,\ldots,m$) is Lipschitz continuous with constant $L_i$, the corresponding $A_i$ is full row rank and the corresponding $\X_i=\Re^{n_i}$". The optimality conditions for the $x_i$- and $ x_i^*$-subproblems are respectively:
\begin{eqnarray}\nonumber
A_i^\top z^{k+1}= \nabla\theta_i(x_i^{k+1})+\beta A_i^\top \sum_{j=i+1}^m(A_j x_j^{k}-A_j x_j^{k+1}),\end{eqnarray}
and
\begin{eqnarray}\nonumber
-A_i^\top  z^*= -\nabla\theta_i(x_i^*).\end{eqnarray}
Adding these two equalities, we get
\begin{eqnarray}
\label{B21}
&& \sqrt{\lambda_{min}(A_iA_i^\top )}\|z^{k+1}-z^*\|\le\|A_i^\top  z^{k+1}-A_i^\top z^*\| \le\|\nabla\theta_i(x_i^{k+1})-\nabla\theta_i(x_i^*)\|+\beta \|A_i\|\sum_{j=i+1}^m\|A_jx_j^{k}- A_jx_j^{k+1}\|\nonumber\\
 &&\le L_i\|  x_i^{k+1}- x^*\|
 +\beta \|A_i\|\sum_{j=i+1}^m\|A_jx_j^{k}- A_j x_j^{k+1}\|,\;\;i=2,\ldots,m.
\end{eqnarray}
Then, there exists a constant $\sigma_1>0$ such that conclusion (\ref{LB1}) follows.

\indent{\bf Case II}: ``$\nabla\theta_1$ is Lipschitz continuous with constant $L_1$,  $A_1$ is nonsingular and $\X_1=\Re^{n_1}$".  It follows from (\ref{lap}) and the last equation in (\ref{MA-VIxyl}) that
\begin{eqnarray}\nonumber
A_1x_1^{k+1}-A_1 x_1^*=\frac{1}{\beta}(z^k-z^{k+1})-\sum_{i=2}^m(A_i x_i^{k+1}-A_i x_i^*).\end{eqnarray}
Then, because $A_1$ is nonsingular, we have
\begin{eqnarray*}
\sqrt{\lambda_{min}(A_1^\top A_1)}\|x_1^{k+1}-x_1^*\| &\le&\frac{1}{\beta}\|z^k-z^{k+1}\|+\sum_{i=2}^m\|A_i x_i^{k+1}-A_i x_i^* \|\nonumber\\
&\le&\frac{1}{\beta}\|z^k-z^{k+1}\|+\sum_{i=3}^m\|A_i x_i^{k+1}-A_i x_i^* \|+\sqrt{\|A_2^\top A_2\|}\|x_2^{k+1}-x_2^*\|.
\end{eqnarray*}
On the other hand, similarly as (\ref{B21}), we get
\begin{eqnarray}\nonumber
\sqrt{\lambda_{min}(A_1A_1^\top )}\|z^{k+1}-z^*\| \le L_1\|  x_1^{k+1}- x^*\|
 +\beta \|A_1\|\sum_{j=2}^m\|A_j x_j^{k}- A_j x_j^{k+1}\|.\end{eqnarray}
Combining these two inequalities, the conclusion (\ref{LB1}) follows immediately.
\end{proof}

Now, we can derive the asymptotically linear convergence of the e-ADMM \eqref{EADMM} under Assumptions \ref{ASm}. To compare with Theorems 3.1-3.3 in \cite{LMZb}, we just show the linear convergence of the sequence $\{(A_1x_1^{k+1},A_2 x_2^{k+1}, \ldots,A_mx_m^{k+1},z^{k+1})\}$. If further assumptions are assumed such as that all $A_i$ $(i=1,\cdots,m)$ are assumed to be full column rank, it is trivial to derive the linear convergence of
the sequence $\{(x_1^{k+1}, x_2^{k+1}, \ldots,x_m^{k+1},z^{k+1})\}$. We skip the detail for succinctness.

\begin{theorem}\label{LR1}
Suppose Assumptions \ref{ASm}. Let $\{(x_1^{k+1}, x_2^{k+1}, \ldots,x_m^{k+1},z^{k+1})\}$ be the sequence generated by the e-ADMM (\ref{EADMM}) with $m\ge 3$ and the restriction of $\beta$ (\ref{betrangup}). Then the sequence $\{(A_1x_1^{k+1},A_2 x_2^{k+1},$ $\ldots,A_mx_m^{k+1},z^{k+1})\}$ converges linearly to a point in $\{(A_1x_1^*,A_2 x_2^*, \ldots,A_mx_m^*,z^*)|w^*\in{\cal W}^*\}$.
\end{theorem}
\begin{proof}
Let $w^*$ be a saddle point in ${\cal W}^*$. It follows from \eqref{LB1} that there exists a positive scalar $\sigma'$ such that
\begin{eqnarray}\nonumber
&&\Phi(v^{k+1},v^k,v^*)\nonumber\\
&&\le\sigma'\left(\sum_{i=2}^m \|A_i x_i^{k+1}-A_i x_i^k\|^2+\|z^k-z^{k+1}\|^2+\sum_{i=3}^m\|A_i x_i^{k+1}-A_i x_i^*\|^2 +\|x_2^{k+1}-x_2^*\|^2+\|x_2^{k+1}-x_2^k\|^2\right).\nonumber
\end{eqnarray}
Then, combining \eqref{K1} and the above inequality, we obtain
\begin{eqnarray*}
\Phi(v^{k+1},v^k,v^*)\le\frac{\sigma'}{\zeta+\sigma'}\Phi(v^{k},v^{k-1},v^*).\end{eqnarray*}
It implies the Q-linearly convergence rate of the sequence $\{\Phi(v^{k+1},v^k,v^*)\}$.
 Thus, we know that the sequences $\{\|A_i x_i^{k+1}-A_i x_i^*\|^2\}$ $(i=2,\ldots,m)$, $\{\|z^{k+1}-z^*\|^2\}$ and $\{\|A_i x_i^{k+1}-A_i x_i^k\|^2\}$  $(i=3,\ldots,m)$ all converges R-linearly.
 Recall that
 \begin{eqnarray}\nonumber
 \|z^k-z^{k+1}\|^2\le2(\|z^k-z^*\|^2+\|z^{k+1}-z^*\|^2).
 \end{eqnarray}
The sequence $\|z^k-z^{k+1}\|^2$ also converges R-linearly. Finally, it follows from (\ref{lap}) that
 \begin{eqnarray}\nonumber
 A_1x_1^{k+1}-A_1 x_1^*=\frac{1}{\beta}(z^k-z^{k+1})-\sum_{i=2}^m(A_i x_i^{k+1}-A_i x_i^* ).
 \end{eqnarray}
 Then, using Cauchy-Schwarz inequality, we have
 \begin{eqnarray}
 \nonumber
 \|A_1x_1^{k+1}-A_1 x_1^*\|^2\le 2\left(\frac{1}{\beta^2}\|z^k-z^{k+1}\|^2+(m-1)\sum_{i=2}^m\|A_i x_i^{k+1}-A_i x_i^*\|^2\right).
 \end{eqnarray}
 Therefore, the sequence $\{\|A_1x_1^{k+1}-A_1 x_1^*\|^2\}$ also converges R-linearly because of the R-linear convergence of the sequences $\{\|z^k-z^{k+1}\|^2\}$ and $\{\|A_i x_i^{k+1}-A_i x_i^*\|^2\}$ $(i=2,\ldots,m)$.
 The proof is complete.
\end{proof}


\begin{table}
\begin{center}
{\scriptsize \caption{Comparison of assumptions and restrictions on $\beta$ with \cite{LMZb}}\label{Table1.1}}
 {\scriptsize \vskip 0mm
\begin{tabular}{|c|l|l|}
   \hline
{Scenario}&{ Assumptions in \cite{LMZb}}&{Assumption \ref{ASm}}\\
\cline{2-3}
&   $\X_i=\Re^{n_i}$ $(i=1,\ldots,n)$&    General nonempty closed convex sets    \\
\hline
1.        &$\theta_2,\cdots,\theta_m$ are strongly convex;  & $\theta_2,\cdots,\theta_m$ are strongly convex;\\
          &$\nabla \theta_m$ is Lipschitz continuous        & One of $\nabla \theta_i$ $(i=2,\cdots,m)$ is Lipschitz continuous,  \\
          &$A_m$ is full row rank                           & and $A_i$ is full row rank with $\X_i=\Re^{n_i}$.\\
          &$\left(0,\min\left\{\min_{2\le i\le m-1}\frac{4\mu_i}{(2m-i)(i-1)\|A_i^\top A_i\|},\frac{4\mu_m}{(m+1)(m-2)\|A_m^\top A_m\|} \right\}\right)$&$\left(0,\;
\min_{3\le i\le m} \{\frac{\mu_i}{\max\{4m-10,3m-6.5\}\|A_i^\top A_i\|}\}\right)$\\
          &&\\
          &&\\
          \hline
2.        & $\theta_1,\cdots,\theta_m$ are strongly convex; & $\theta_2,\cdots,\theta_m$ are strongly convex;\\
          & $\nabla\theta_1,\cdots,\nabla\theta_m$  are Lipschitz continuous  & $\nabla \theta_1$  is Lipschitz continuous, \\
          &&and $A_1$ is nonsingular with $\X_i=\Re^{n_i}$\\
          &$\left(0,\min\left\{\min_{2\le i\le m-1}\frac{4\mu_i}{3(2m-i)(i-1)\|A_i^\top A_i\|},\frac{4\mu_m}{3m^2-3m-2\|A_m^\top A_m\|} \right\}\right)$  & $\left(0,\;
\min_{3\le i\le m} \{\frac{\mu_i}{\max\{4m-10,3m-6.5\}\|A_i^\top A_i\|}\}\right)$\\
          &&\\
          &&\\
          \hline
3.        & $\theta_2,\cdots,\theta_m$ are strongly convex;&\\
          & $\nabla\theta_1,\cdots,\nabla\theta_m$  are Lipschitz continuous  &\\
          & $A_1$ is full column rank.&\\
          & $\left(0,\min\left\{\min_{2\le i\le m-1}\frac{4\mu_i}{3(2m-i)(i-1)\|A_i^\top A_i\|},\frac{4\mu_m}{3m^2-3m-2\|A_m^\top A_m\|} \right\}\right)$  &\\
\hline
\end{tabular}}
\end{center}
\end{table}

\begin{remark}
In \cite{LMZb}, three scenarios are considered to ensure the linear convergence of e-ADMM (\ref{EADMM}). We list them in Table \ref{Table1.1}. Note that all the cases in \cite{LMZb} additionally require ${\cal X}_i={\Re}^{n_i}$ $(i=1,\ldots,m)$. Scenario 1 in Table 2 of \cite{LMZb} is included in our Assumption \ref{ASm}; while we can easily establish the linear convergence of the e-ADMM with (\ref{betrangup}) for Scenarios 2 and 3 in Table 2 of \cite{LMZb} by following the roadmap of the proof of Theorem \ref{LR1}. For succinctness, we omit the proof details for Scenarios 2 and 3  in Table 2 of \cite{LMZb}.
Now we elaborate on the difference of the restrictions on $\beta$ in Table \ref{Table1.1}. Note that the denominator of the upper bound for $\beta$ in (\ref{betrangup}) is a linear function of $m$ while that in \cite{LMZb} is  quadratic. So it is not hard to see that our restriction of $\beta$ (\ref{betrangup}) is less restrictive than those in \cite{LMZb}. More specifically, for example, if we consider the case of $m=15$ and $\mu_i\equiv\mu$ ($i=2,\ldots,m$), then the ranges of $\beta$ in \cite{LMZb} for Scenarios 1 and 2 are $(0,\frac{\mu}{52})$ and $(0,\frac{\mu}{157})$, respectively; while those in (\ref{betrangup}) for both cases are $(0,\frac{\mu}{50})$. The difference of $\beta$'s range becomes more apparent for larger values of $m$.
\end{remark}

\section{Two assertions}\label{ext}

In this section, we construct some examples to show that the e-ADMM (\ref{EADMM}) with $m\ge 3$ are divergent if the model (\ref{mini3}) or the penalty parameter $\beta$ in (\ref{EADMM}) is not appropriately assumed. These examples exclude the hope of ensuring the convergence of (\ref{EADMM}) under too mild assumptions and to some extent justify the rationale of our assumptions for discussing the convergence of the e-ADMM (\ref{EADMM}) with $m\ge 3$. In particular, we verify the following assertions.

\begin{itemize}
\item The e-ADMM  may be divergent for solving (\ref{mini3}) if $(m-3)$ functions are strongly convex for any penalty parameter $\beta>0$ with $m\ge4$;
\item The e-ADMM  may be divergent for solving (\ref{mini3}) if $(m-2)$ functions are strongly convex without any restriction on the penalty parameter $\beta$ with $m\ge3$.
\end{itemize}

\subsection{Application of e-ADMM to a linear homogeneous equation}

Inspired by \cite{CHYY}, we consider a linear homogeneous equation with $m$ variables
\begin{equation}
\label{EQFour}
\sum_{i=1}^m A_i x_i=0,
\end{equation}
where $A_i\in\Re^m$ $(i=1,\ldots,m)$ are all column vectors and the matrix $A:=[A_1,\ldots,A_m]$ is assumed to be nonsingular. Obviously, the equation (\ref{EQFour}) has the unique solution $x_i={ 0}$ $(i=1,\ldots,m)$.
The equation (\ref{EQFour}) is a special case of the model (\ref{mini3}) where the objective function is null, $b$ is the all-zero vector in $\Re^m$ and ${\cal X}_i =\Re$ for $i=1,\ldots,m$.

Applying the e-ADMM (\ref{EADMM}) with $\beta>0$ to (\ref{EQFour}), we obtain
\begin{eqnarray}\label{EADMMEQ}
\left\{
\begin{array}{l}
-A_1^\top z^k+\beta A_1^\top (A_1 x_1^{k+1}+A_2 x_2^{k}+\cdots+A_m x_m^k)=0,\\
-A_2^\top z^k+\beta A_2^\top (A_1 x_1^{k+1}+A_2 x_2^{k+1}+\cdots+A_m x_m^k)=0,\\
\qquad\qquad\cdots\cdots\quad\cdots\cdots\quad\cdots\cdots\\
-A_m^\top z^k+\beta A_m^\top (A_1 x_1^{k+1}+A_2 x_2^{k+1}+\cdots+A_m x_m^{k+1})=0,\\
z^{k+1}=z^k-\beta(A_1 x_1^{k+1}+A_2 x_2^{k+1}+\cdots+A_m x_m^{k+1}).
\end{array}
\right.
\end{eqnarray}
Introducing the new variable $\mu^k:=z^k/\beta$, the scheme (\ref{EADMMEQ}) can be rewritten as
\begin{eqnarray}\label{iterM}
\left(\begin{array}{c}
x_2^{k+1}\\
\vdots\\
x_m^{k+1}\\
\mu^{k+1}
\end{array}
\right)=S
\left(\begin{array}{c}
x_2^{k}\\
\vdots\\
x_m^{k}\\
\mu^{k}
\end{array}
\right),\;\mbox{where}\; S=L^{-1}R,
\end{eqnarray}
with
\begin{eqnarray}
\label{matrixL}
L=\left(\begin{array}{cccccc}
A_2^\top A_2 &0           &\cdots &0         &0  &{\bf 0}_{1\times m}\\
A_3^\top A_2 &A_3^\top A_3&\cdots &0         &0  &{\bf 0}_{1\times m}\\
\vdots&\vdots&\ddots&\vdots&\vdots&\vdots\\
A_{m-1}^\top A_2 &\cdots&\cdots&A_{m-1}^\top A_{m-1}&0&{\bf 0}_{1\times m}\\
A_m^\top A_2&A_m^\top A_3&\cdots&A_m^\top A_{m-1}&A_{m}^\top A_m&{\bf 0}_{1\times m}\\
A_2&          A_3 &\cdots        &A_{m-1}&A_m         &I_{m\times m}
\end{array}
\right)
\end{eqnarray}
and
\begin{eqnarray}
\label{matrixR}
&&R=\left(\begin{array}{cccccc}
0&-A_2^\top A_3&\cdots&\cdots&-A_2^\top A_m  &A_2^\top\\
0&0            &\cdots&\cdots&-A_3^\top A_m  &A_3^\top\\
\vdots&\vdots&\ddots&\vdots&\vdots&\vdots\\
0&0            &\cdots &\ddots&-A_{m-1}^\top A_m &A_{m-1}^\top\\
0&0&\cdots&\cdots&0&A_m^\top\\
{\bf 0}_{m\times 1}  &{\bf 0}_{m\times 1}  &{\bf 0}_{m\times1}&\cdots&{\bf 0}_{m\times1}&I_{m\times m}
\end{array}
\right) \\
&&-\frac{1}{A_1^\top A_1}\left(\begin{array}{c}A_2^\top A_1\\A_3^\top A_1\\\vdots\\A_{m-1}^\top A_1\\ A_m^\top A_1\\A_1
\end{array}
 \right)\left( \begin{array}{c}-A_1^\top A_2,-A_1^\top A_3,\cdots,-A_{1}^\top A_{m-1},-A_1^\top A_m,A_1^\top
\end{array}\right).\nonumber\\
\end{eqnarray}
Therefore, the e-ADMM (\ref{EADMM}) for solving (\ref{EQFour}) is divergent if the spectral radius of $S$, denoted by $\rho(S)$,  is strictly larger than 1.
Note that $\rho(S)$ is independent of $\beta$. That is, when the e-ADMM (\ref{EADMM}) is applied to the special problem (\ref{EQFour}), the convergence is independent of the value of $\beta$.

Based on the analysis above, the divergence of (\ref{EADMM}) with $m=3$ for any $\beta>0$ has been illustrated in \cite{CHYY} by the example with $A$ defined as
\begin{eqnarray}\label{EQFour-1}
A=(A_1,A_2,A_3)=\left(
\begin{array}{llll}
1&1&1\\
1&1&2\\
1&2&2\\
\end{array}
\right).
\end{eqnarray}
For this case, we have $\rho(S)=  1.0278>1$ where $S$ is the corresponding matrix given in (\ref{iterM}).
We can extend the assertion to the more general case of $m \ge 3$. Indeed, the following theorem can be easily proved by mathematical induction; thus we omit the proof. But the same technique will be used for constructing other examples.

\begin{theorem}\label{theorem71}
For model (\ref{mini3}) with $m\ge 3$, the e-ADMM (\ref{EADMM}) is not necessarily convergent for any $\beta>0$.
\end{theorem}
\begin{proof}Indeed, we can use mathematical induction on $m$ to show that for any $m\ge3$, there exists one specific matrix $ A^{(m)}\in\Re^{m\times m}$ such that
 the corresponding iterative matrix, i.e., $S$ in (\ref{iterM}), satisfies $\rho(S)>1$ when e-ADMM is applied to (\ref{EQFour}) with $A:= A^{(m)}$.
\end{proof}

\subsection{The divergence of (\ref{EADMM}) with $(m-3)$ strongly convex functions for any $\beta>0$ with $m\ge 4$}
\label{sect72}
Let us consider the problem
\begin{eqnarray}
\label{Fourstrcov}
\begin{array}{cl}
\min_x& 0.05 x_4^2\\
s.t.&\left( \begin{array}{llll}
1&1&1&0\\
1&1&2&0\\
1&2&2&0\\
0&0&0&1\end{array}\right)\left(\begin{array}{c}
x_1\\
x_2\\
x_3\\
x_4
\end{array}\right)=0
\end{array}
\end{eqnarray}
which corresponds to the model (\ref{mini3}) with one strongly convex function in its objective.
Let us denote the coefficient matrix of the linear constraint in (\ref{Fourstrcov}) by $\hat A: =[\hat A_1,\hat A_2,\hat A_3,\hat A_4]$. Applying the e-ADMM (\ref{EADMM}) with $\beta>0$ to (\ref{Fourstrcov}),
we obtain
\begin{eqnarray}\nonumber
\left\{
\begin{array}{l}
-{\hat A}_1^\top \mu^k+ {\hat A_1}^\top ({\hat A_1} x_1^{k+1}+{\hat A}_2 x_2^{k}+{\hat A}_3 x_3^{k}+{\hat A}_4 x_4^k)=0,\\
-{\hat A}_2^\top \mu^k+{\hat A}_2^\top ({\hat A}_1 x_1^{k+1}+{\hat A}_2 x_2^{k+1}+{\hat A}_3 x_3^{k}+{\hat A}_4 x_4^k)=0,\\
-{\hat A}_3^\top \mu^k+{\hat A}_3^\top ({\hat A}_1 x_1^{k+1}+{\hat A}_2 x_2^{k+1}+{\hat A}_3 x_3^{k+1}+{\hat A}_4 x_4^k)=0,\\
-{\hat A}_4^\top \mu^k+ {\hat A}_4^\top ({\hat A}_1 x_1^{k+1}+{\hat A}_2 x_2^{k+1}+{\hat A}_3 x_3^{k+1}+{\hat A}_4 x_4^{k+1})+\frac{0.1}{\beta} x_4^{k+1}=0,\\
\mu^{k+1}=\mu^k-({\hat A}_1 x_1^{k+1}+{\hat A}_2 x_2^{k+1}+{\hat A}_3 x_3^{k+1}+{\hat A}_4 x_4^{k+1}),
\end{array}
\right.
\end{eqnarray}
which depends on the penalty parameter $\beta$.
Note ${\hat A}_4^\top {\hat A}_i=0$ $(i=1,2,3)$. Consequently, we have
\begin{eqnarray}\nonumber
\left\{
\begin{array}{l}
-{\hat A}_1^\top \mu^k+ {\hat A_1}^\top ({\hat A_1} x_1^{k+1}+{\hat A}_2 x_2^{k}+{\hat A}_3 x_3^{k})=0,\\
-{\hat A}_2^\top \mu^k+{\hat A}_2^\top ({\hat A}_1 x_1^{k+1}+{\hat A}_2 x_2^{k+1}+{\hat A}_3 x_3^{k})=0,\\
-{\hat A}_3^\top \mu^k+{\hat A}_3^\top ({\hat A}_1 x_1^{k+1}+{\hat A}_2 x_2^{k+1}+{\hat A}_3 x_3^{k+1})=0,\\
x_4^{k+1}={\hat A}_4^\top \mu^k/(1+\frac{0.1}{\beta}),\\
\mu^{k+1}=\mu^k-({\hat A}_1 x_1^{k+1}+{\hat A}_2 x_2^{k+1}+{\hat A}_3 x_3^{k+1}+{\hat A}_4 x_4^{k+1}).
\end{array}
\right.
\end{eqnarray}
Moreover, if we set ${\hat \mu}:=\mu_{[1:3]}$, and recall the definition of $A$ in (\ref{EQFour-1}), we get
\begin{eqnarray}\label{tmp1-D}
\left\{
\begin{array}{l}
-{A}_1^\top {\hat \mu}^k + {A}_1^\top ({ A_1} x_1^{k+1}+{ A}_2 x_2^{k}+{ A}_3 x_3^{k})=0,\\
-{A}_2^\top {\hat \mu}^k +{A}_2^\top ({A}_1 x_1^{k+1}+{ A}_2 x_2^{k+1}+{ A}_3 x_3^{k})=0,\\
-{A}_3^\top {\hat \mu}^k +{A}_3^\top ({A}_1 x_1^{k+1}+{ A}_2 x_2^{k+1}+{ A}_3 x_3^{k+1})=0,\\
x_4^{k+1}=\mu_{[4]}^k/(1+\frac{0.1}{\beta}),\\
{\hat \mu}^{k+1}={\hat \mu}^k-({ A}_1 x_1^{k+1}+{ A}_2 x_2^{k+1}+{ A}_3 x_3^{k+1}),\\
\mu^{k+1}_{[4]}=\mu^k_{[4]}-x_4^{k+1}.
\end{array}
\right.
\end{eqnarray}

Also, we denote $y^\top:=(x_2^\top,x_3^\top,{\hat \mu}^\top)$. Then, the iterative scheme (\ref{tmp1-D}) is
\begin{eqnarray}\label{tmp2-D}
\left(\begin{array}{l}
y^{k+1}\\
\mu_{[4]}^{k+1}\\
x_4^{k+1}\end{array}\right)=\left(\begin{array}{lll}
S&{\bf 0}_{5 \times 1}&{\bf 0}_{5 \times 1}\\
{\bf 0}_{1 \times 5}&1&-1\\
{\bf 0}_{1 \times 5}&1/(1+\frac{0.1}{\beta})&0
\end{array}
\right)
\left(\begin{array}{l}
y^{k}\\
\mu_{[4]}^{k}\\
x_4^{k}\end{array}\right).
\end{eqnarray}
where $S$ is the matrix given in (\ref{iterM}) when the e-ADMM is applied to (\ref{EQFour}) with $A$ defined in (\ref{EQFour-1}).
Let $\hat S(\beta)$ be the coefficient matrix given in (\ref{tmp2-D}), which is clearly dependent on $\beta$. Then, we have $\rho(\hat S(\beta))\ge\rho(S)=1.0278>1$ for any $\beta>0$.
 In fact, we have $\rho(\hat S(\beta))\ge\rho(S)$ because the absolute value of the maximum eigenvalue of the $2\times 2$ submatrix in the lower right corner of the coefficient matrix in (\ref{tmp2-D}) is no big than 1. Hence, the e-ADMM (\ref{EADMM}) may be divergent for any penalty parameter $\beta>0$ if $m=4$ and there is one strongly convex function. We extend the conclusion to the general case of $m\ge 4$ in the following theorem.

\begin{theorem}\label{theorem72}
For model (\ref{mini3}) with $m\ge 4$, the e-ADMM (\ref{EADMM}) is not necessarily convergent for an arbitrarily fixed $\beta>0$ if there are $(m-3)$ strongly convex functions in the objective of (\ref{mini3}).
\end{theorem}

\begin{proof}
For any $m\ge4$, we consider the following convex programming:
\begin{eqnarray}\label{barAM}
\min& \frac{\sigma}{2}\sum_{i=4}^m x_i^2\nonumber\\
s.t.& \sum_{i=1}^m \check{A}_i x_i=0,
\end{eqnarray}
where $\check{A}:=[\check{A}_1,\ldots,\check{A}_m]\in\Re^{m\times m}$ or rewritten as
\begin{eqnarray}
\nonumber
\check{A}=\left(\begin{array}{ll}
A&{\bf 0}_{3\times (m-3)}\\
{\bf 0}_{(m-3)\times 3}&{\bf I}_{(m-3)\times (m-3)}
\end{array}\right),
\end{eqnarray}
where the $3\times 3$ matrix $A$ is defined in (\ref{EQFour-1}).
Applying the e-ADMM (\ref{EADMM}) with $\beta>0$ to (\ref{barAM}),
we obtain
\begin{eqnarray}\nonumber
\left\{
\begin{array}{l}
-{\check A}_1^\top \mu^k+ {\check A_1}^\top ({\check A_1} x_1^{k+1}+{\check A}_2 x_2^{k}+\cdots+{\check A}_m x_m^k)=0,\\
-{\check A}_2^\top \mu^k+{\check A}_2^\top ({\check A}_1 x_1^{k+1}+{\check A}_2 x_2^{k+1}+\cdots+{\check A}_m x_m^k)=0,\\
-{\check A}_3^\top \mu^k+{\check A}_3^\top ({\check A}_1 x_1^{k+1}+{\check A}_2 x_2^{k+1}+\cdots+{\check A}_m x_m^k)=0,\\
\qquad\cdots\cdots\quad\cdots\cdots\quad\cdots\cdots\quad\cdots\cdots\quad\cdots\cdots\qquad\\
-{\check A}_m^\top \mu^k+ {\check A}_m^\top ({\check A}_1 x_1^{k+1}+{\check A}_2 x_2^{k+1}+\cdots+{\check A}_m x_m^{k+1})+\frac{\sigma_m}{\beta} x_m^{k+1}=0,\\
\mu^{k+1}=\mu^k-({\check A}_1 x_1^{k+1}+{\check A}_2 x_2^{k+1}+{\check A}_3 x_3^{k+1}+{\check A}_4 x_4^{k+1}).
\end{array}
\right.
\end{eqnarray}
Note that we have
\begin{eqnarray*}
{\check A}_j^\top {\check A}_i=0, \;when\; \left\{\begin{array}{l}
                                           i=1,2,3\;\mbox{and}\;j>3, \\
                                           i,j>3 \;\mbox{and}\;  i\neq j,\\
                                           j=1,2,3\;\mbox{and}\;i>3.
                                           \end{array}\right.
                                           \end{eqnarray*}
Consequently, it implies that
\begin{eqnarray}\nonumber
\left\{
\begin{array}{l}
-{\check A}_1^\top \mu^k+ {\check A_1}^\top ({\check A_1} x_1^{k+1}+{\check A}_2 x_2^{k}+{\check A}_3 x_3^{k})=0,\\
-{\check A}_2^\top \mu^k+{\check A}_2^\top ({\check A}_1 x_1^{k+1}+{\check A}_2 x_2^{k+1}+{\check A}_3 x_3^{k})=0,\\
-{\check A}_3^\top \mu^k+{\check A}_3^\top ({\check A}_1 x_1^{k+1}+{\check A}_2 x_2^{k+1}+{\check A}_3 x_3^{k+1})=0,\\
x_i^{k+1}={\check A}_i^\top \mu^k/(1+\frac{\sigma_i}{\beta}),\;i=4,\ldots,m,\\
\mu^{k+1}=\mu^k-({\check A}_1 x_1^{k+1}+{\check A}_2 x_2^{k+1}+{\check A}_3 x_3^{k+1}+\cdots+{\check A}_m x_m^{k+1}).
\end{array}
\right.
\end{eqnarray}
Moreover, setting ${\check \mu}:=\mu_{[1:3]}$ and recall the definition of $A$ in (\ref{EQFour-1}), we get
\begin{eqnarray}\label{tmp1-D-new}
\left\{
\begin{array}{l}
-{A}_1^\top {\hat \mu}^k + {A}_1^\top ({ A_1} x_1^{k+1}+{ A}_2 x_2^{k}+{ A}_3 x_3^{k})=0,\\
-{A}_2^\top {\check \mu}^k +{A}_2^\top ({A}_1 x_1^{k+1}+{ A}_2 x_2^{k+1}+{ A}_3 x_3^{k})=0,\\
-{A}_3^\top {\check \mu}^k +{A}_3^\top ({A}_1 x_1^{k+1}+{ A}_2 x_2^{k+1}+{ A}_3 x_3^{k+1})=0,\\
x_i^{k+1}= \mu_{[i]}^k/(1+\frac{\sigma_i}{\beta}),\;i=4,\ldots,m,\\
{\check \mu}^{k+1}={\check \mu}^k-({ A}_1 x_1^{k+1}+{ A}_2 x_2^{k+1}+{ A}_3 x_3^{k+1}),\\
\mu^{k+1}_{[i]}=\mu^k_{[i]}-x_i^{k+1}\;i=4,\ldots,m.
\end{array}
\right.
\end{eqnarray}
Let us denote $y^\top:=(x_2^\top,x_3^\top,{\check \mu}^\top)$. Then, the iterative scheme (\ref{tmp1-D}) can be written as
\begin{eqnarray}\label{tmp2-D-new}
\left(\begin{array}{l}
y^{k+1}\\
\mu_{[4]}^{k+1}\\
\vdots\\
\mu_{[m]}^{k+1}\\
x_4^{k+1}\\
\vdots\\
x_m^{k+1}\end{array}\right)=\left(\begin{array}{lll}
S&{\bf 0}_{3 \times {(m-3)}}&{\bf 0}_{3 \times {(m-3)}}\\
{\bf 0}_{(m-3)\times 3}&I_{(m-3)\times(m-3)}&-I_{(m-3)\times(m-3)}\\
{\bf 0}_{(m-3) \times 3}&D_{(m-3)\times(m-3)}&{\bf 0}_{(m-3) \times (m-3)}
\end{array}
\right)
\left(\begin{array}{l}
y^{k}\\
\mu_{[4]}^{k}\\
\vdots\\
\mu_{[m]}^{k}\\
x_4^{k}\\
\vdots\\
x_m^{k}\end{array}\right),
\end{eqnarray}
where
\begin{eqnarray}\nonumber
D_{(m-3)\times(m-3)}=\hbox{diag}(\frac{1}{1+\frac{\sigma_4}{\beta}},\cdots,\frac{1}{1+\frac{\sigma_m}{\beta}}).
\end{eqnarray}
It can be easily shown that the absolution value of the maximal eigenvalue of
\begin{eqnarray}
\left(\begin{array}{cc}
I_{(m-3)\times(m-3)}&-I_{(m-3)\times(m-3)}\\
D_{(m-3)\times(m-3)}&{\bf 0}_{(m-3) \times (m-3)}
\end{array}\right)
\end{eqnarray}
is less than 1. Thus, for the coefficient matrix in (\ref{tmp2-D-new}), denoted by ${\check {S}}$, we have
$$
\rho(\check{ S}(\beta))\ge\rho(S)=1.0278>1, \;\;\forall \beta>0.
$$
The proof is complete.
\end{proof}

\subsection{The divergence of (\ref{EADMM}) with $(m-2)$ strongly convex functions and without restriction on $\beta$ for $m\ge 3$}
\label{sect73}
Then, we show that the e-ADMM (\ref{EADMM}) may be divergent even if there are $(m-2)$ strongly convex functions in the objective of (\ref{mini3}) while there is no restriction on $\beta$ for the case of $m\ge3$. We consider the model
\begin{eqnarray}
\label{abstrfour}
\begin{array}{cl}
\min& \frac{1}{2}a_3x_3^2\\
s.t.&A_1 x_1+A_2 x_2+A_3 x_3=0,
\end{array}
\end{eqnarray}
where $A_i\in\Re^3$ $(i=1,2,3)$ are all column vectors, the matrix $A:=[A_1,A_2,A_3]$ is assumed to be nonsingular,
and the scalar $a_3$ is positive. It is easy to see that each iteration of (\ref{EADMM}) applied to (\ref{abstrfour})
can be characterized by a matrix iteration (\ref{iterM}) with the iterative matrix $\hat S$ defined as:
\begin{eqnarray}\label{martix-2example}
\hat S={\hat L}^{-1} R,
\end{eqnarray}
where the matrix $R$ is defined in (\ref{matrixR}) with $m=3$, and $\hat L$ is defined below:
\begin{eqnarray}
\label{matrixtildeL}
\hat L=\left(\begin{array}{lll}
A_2^\top A_2 &0                    &0_{1\times 3}\\
A_3^\top A_2 &A_3^\top A_3+a_3/\beta          &0_{1\times 3}\\
A_2&          A_3                &I_{3\times 3}
\end{array}
\right).
\end{eqnarray}
Specifically, let us take
\begin{eqnarray}\label{Ex2-A}
A=[A_1,A_2,A_3]=\left(
\begin{array}{lll}
1&1&1\\
1&1&2\\
1&2&2\\
\end{array}\right)\;\mbox{and}\;a_3=0.05.
\end{eqnarray}
With simple calculations, if we take $\beta=1$, then we have $\rho(\hat S)$=1.0259 for the matrix $\hat S$ given in (\ref{martix-2example}). On the other hand, for this example, we have $m=3$, $f_3(x_3)=0.025x_3^2$ and $\mu_3=0.05$. According to Theorem \ref{Thm-convergence}, the e-ADMM (\ref{EADMM}) is guaranteed to be convergent when $\beta\in(0,\frac{0.05}{7.5})$, i.e., $\rho(\hat M)=0.9586$ when $\beta=0.0066\in(0,\frac{0.05}{7.5})$. Now, we extend the conclusion to the general case of $m\ge 3$. This theorem also shows that it is necessary to appropriately restrict the penalty parameter $\beta$ when discussing the convergence of the e-ADMM (\ref{EADMM})

\begin{theorem}\label{theorem73}
For model (\ref{mini3}) with $m\ge3$, the e-ADMM (\ref{EADMM}) is not necessarily convergent for all $\beta>0$ when there are $(m-2)$ strongly convex functions in its objective.
\end{theorem}
\begin{proof}
In the following, we show that for any $\beta>0$, there exist examples such that the e-ADMM (\ref{EADMM}) is divergent. First, it follows from Theorem \ref{theorem71} that there exists a specific matrix $A^{(m)}\in\Re^{m\times m}$ with $m\ge3$ such that the e-ADMM (\ref{EADMM}) is divergent when it is applied to the equation (\ref{EQFour}) with $A:=A^{(m)}$. It implies that the corresponding matrix $S^{(m)}$ given in (\ref{iterM}) has $\rho(S^{(m)})>1$. Recall the matrices $L$ and $R$ composing $S^{(m)}$ by $S^{(m)}=L^{-1}R$ (see (\ref{matrixL}) and (\ref{matrixR})). Then, with this specific choice of $A^{(m)}$, we consider the following problem:
\begin{eqnarray}
\label{EQmonestrC}
\begin{array}{ll}
\min&\frac{\sigma}{2}\sum_{i=3}^m x_i^2\\
s.t.&\sum_{i=1}^m A^{(m)}_i x_i=0,
\end{array}
\end{eqnarray}
where $A^{(m)}=[A^{(m)}_1,\ldots,A^{(m)}_m]$,
$\sigma>0$ and there are $(m-2)$ strongly convex functions in its objective.
One can show that each iteration of (\ref{EADMM}) applied to (\ref{EQmonestrC})
can be characterized by a matrix iteration (\ref{iterM}) with the iterative matrix ${\tilde S}^{(m)}$
defined as  ${\tilde S}^{(m)}=\tilde L^{-1}R$ with $R$ define in (\ref{matrixR}) and $\tilde L$ defined by
\begin{eqnarray}
\nonumber
{\tilde L}=\left(\begin{array}{cccccc}
A_2^\top A_2 &0           &\cdots &0        &0   &{\bf 0}_{1\times m}\\
A_3^\top A_2 &A_3^\top A_3+\sigma/\beta&\cdots &0       &0    &{\bf 0}_{1\times m}\\
\vdots&\vdots&\ddots&\vdots&\vdots&\vdots\\
A_{m-1}^\top A_2 &\cdots&\cdots&A_{m-1}^\top A_{m-1}+\sigma/\beta&0&{\bf 0}_{1\times m}\\
A_m^\top A_2&A_m^\top A_3&\cdots&A_m^\top A_{m-1}&A_m^\top A_m+\sigma/\beta&{\bf 0}_{1\times m}\\
A_2&          A_3 &\cdots       &A_{m-1} &A_m         &I_{m\times m}
\end{array}
\right)
\end{eqnarray}
Set $\beta=1$ in the e-ADMM (\ref{EADMM}).
Then, let $E:={\tilde L}-L=\hbox{diag}({ 0},\sigma,\cdots,\sigma,{\bf 0}_{m\times m})$, we have
\begin{eqnarray}
&&\|I - (I+L^{-1}E)^{-1}\|=\|\sum_{k=1}^{\infty} (-1)^{k+1}(L^{-1}E)^k\|\le\sum_{k=1}^{\infty}\|L^{-1}E\|^k\nonumber\\
&&=\frac{\|L^{-1}E\|}{1-\|L^{-1}E\|}\le \|E\|\|L^{-1}\|={\sigma}\|L^{-1}\|.
\end{eqnarray}
Next, define $\Delta_S:={\tilde S}^{(m)}- S^{(m)}$. Then, we have
\begin{eqnarray}
\label{tmpfinal}
\|\Delta_S\|=\|-(I - (I+L^{-1}E)^{-1})L^{-1}R\|\le\|I - (I+L^{-1}E)^{-1}\|\|L^{-1}R\|\le \sigma\|L^{-1}\|\|L^{-1}R\|.
\end{eqnarray}
Thus, there exist a positive constant $\bar \sigma$ such that $\|\Delta_S\|<1$ when $\sigma\in(0,\bar \sigma]$.
Then, Invoking Lemma \ref{conspec} with setting $A:=S^{(m)}$ and $\Delta:=\Delta_S$, there exists a positive scalar $\kappa$  dependent on $S^{(m)}=L^{-1}R$ such that
\begin{eqnarray}\label{spec}|\rho(S^{(m)})-\rho({\tilde S}^{(m)})|\le \kappa\|\Delta_S\|\le \sigma \kappa\|L^{-1}\|\|L^{-1}R\|,\end{eqnarray}
where the last inequality is due to (\ref{tmpfinal}).
  Then, the right-hand side of the above inequality only depends on $\sigma$ since $\kappa\|L^{-1}\|\|L^{-1}R\|$ is a constant.
Therefore, there exists a sufficient small $\hat\sigma$ $(\hat \sigma\le\bar\sigma)$ such that $\rho(S^{(m)})-{\hat\sigma} \kappa\|L^{-1}\|\|L^{-1}R\|>1$ whenever $\rho(S^{(m)})>1$.
As a consequent, we have $\rho({\tilde S}^{(m)})>1$ due to  (\ref{spec}).
This implies that the e-ADMM (\ref{EADMM}) with $\beta=1$ is divergent when solving (\ref{EQmonestrC}) with setting $\sigma:=\hat \sigma$.
Indeed, for any $\beta>0$, we can construct a specific problem defined in (\ref{EQmonestrC}), i.e., finding a appropriate $\sigma$,  such that the e-ADMM (\ref{EADMM})  with this $\beta$ is divergent. Note (\ref{EQmonestrC}) is a special case of  (\ref{mini3}) with $(m-2)$ strongly convex functions in its objective.
\end{proof}

\section{Conclusions}

In this paper, we conduct convergence analysis for the direct extension of ADMM (``e-ADMM") for solving a separable convex minimization model whose objective function is the sum of $m$ function without coupled variables. We extend the existing result for the special case of $m=3$ to the general case of $m\ge 3$, and prove the convergence of e-ADMM when $(m-2)$ functions are assumed to be strongly convex and the penalty parameter is appropriately restricted. For the special case of $m=3$, our result is still better than some existing ones that are analyzed specifically for this special case of $m=3$ in the sense that the penalty parameter is less restricted. The worst-case convergence rate measured by iteration complexity and asymptotically linear convergence are also derived under some additional assumptions.

\end{document}